\documentclass[11pt, a4paper]{amsart}
\pdfoutput=1 
\usepackage[dvips]{graphicx}
\usepackage{amssymb, amstext, amscd, amsmath, color}
\usepackage{epstopdf}
\usepackage{tikz,mathdots,enumerate}
\usepackage{pgfplots}
\usepackage{url}

\usetikzlibrary{arrows}

\usepackage{kbordermatrix} 
\renewcommand{\kbldelim}{[} 
\renewcommand{\kbrdelim}{]}

\definecolor{cof}{RGB}{219,144,71}
\definecolor{pur}{RGB}{186,146,162}
\definecolor{greeo}{RGB}{91,173,69}
\definecolor{greet}{RGB}{52,111,72}

\usepackage{hhline}

\usepackage{hyperref}

\begin{document}

\newtheorem{theorem}{Theorem}[section]
\newtheorem{corollary}[theorem]{Corollary}
\newtheorem{proposition}[theorem]{Proposition}
\newtheorem{lemma}[theorem]{Lemma}

\theoremstyle{definition}
\newtheorem{example}[theorem]{Example}

\newcommand{\FFock}{\mathcal{F}}
\newcommand{\kil}{\mathsf{k}}
\newcommand{\Hil}{\mathsf{H}}
\newcommand{\hil}{\mathsf{h}}
\newcommand{\Kil}{\mathsf{K}}
\newcommand{\Real}{\mathbb{R}}
\newcommand{\Rplus}{\Real_+}

\newcommand{\bC}{{\mathbb{C}}}
\newcommand{\bD}{{\mathbb{D}}}
\newcommand{\bN}{{\mathbb{N}}}
\newcommand{\bQ}{{\mathbb{Q}}}
\newcommand{\bR}{{\mathbb{R}}}
\newcommand{\bT}{{\mathbb{T}}}
\newcommand{\bX}{{\mathbb{X}}}
\newcommand{\bZ}{{\mathbb{Z}}}
\newcommand{\bH}{{\mathbb{H}}}
\newcommand{\BH}{{\B(\H)}}
\newcommand{\bsl}{\setminus}
\newcommand{\ca}{\mathrm{C}^*}
\newcommand{\cstar}{\mathrm{C}^*}
\newcommand{\cenv}{\mathrm{C}^*_{\text{env}}}
\newcommand{\rip}{\rangle}
\newcommand{\ol}{\overline}
\newcommand{\td}{\widetilde}
\newcommand{\wh}{\widehat}
\newcommand{\sot}{\textsc{sot}}
\newcommand{\wot}{\textsc{wot}}
\newcommand{\wotclos}[1]{\ol{#1}^{\textsc{wot}}}
 \newcommand{\A}{{\mathcal{A}}}
 \newcommand{\B}{{\mathcal{B}}}
 \newcommand{\C}{{\mathcal{C}}}
 \newcommand{\D}{{\mathcal{D}}}
 \newcommand{\E}{{\mathcal{E}}}
 \newcommand{\F}{{\mathcal{F}}}
 \newcommand{\G}{{\mathcal{G}}}
\renewcommand{\H}{{\mathcal{H}}}
 \newcommand{\I}{{\mathcal{I}}}
 \newcommand{\J}{{\mathcal{J}}}
 \newcommand{\K}{{\mathcal{K}}}
\renewcommand{\L}{{\mathcal{L}}}
 \newcommand{\M}{{\mathcal{M}}}
 \newcommand{\N}{{\mathcal{N}}}
\renewcommand{\O}{{\mathcal{O}}}
\renewcommand{\P}{{\mathcal{P}}}
 \newcommand{\Q}{{\mathcal{Q}}}
 \newcommand{\R}{{\mathcal{R}}}
\renewcommand{\S}{{\mathcal{S}}}
 \newcommand{\T}{{\mathcal{T}}}
 \newcommand{\U}{{\mathcal{U}}}
 \newcommand{\V}{{\mathcal{V}}}
 \newcommand{\W}{{\mathcal{W}}}
 \newcommand{\X}{{\mathcal{X}}}
 \newcommand{\Y}{{\mathcal{Y}}}
 \newcommand{\Z}{{\mathcal{Z}}}

\newcommand{\supp}{\operatorname{supp}}
\newcommand{\conv}{\operatorname{conv}}
\newcommand{\cone}{\operatorname{cone}}
\newcommand{\vspan}{\operatorname{span}}
\newcommand{\proj}{\operatorname{proj}}
\newcommand{\sgn}{\operatorname{sgn}}
\newcommand{\rank}{\operatorname{rank}}
\newcommand{\Isom}{\operatorname{Isom}}
\newcommand{\qIsom}{\operatorname{q-Isom}}
\newcommand{\Cknet}{{\mathcal{C}_{\text{knet}}}}
\newcommand{\Ckag}{{\mathcal{C}_{\text{kag}}}}
\newcommand{\rind}{\operatorname{r-ind}}
\newcommand{\lind}{\operatorname{r-ind}}

\setcounter{tocdepth}{1}

 \title[Finite and infinitesimal rigidity with polyhedral norms]{Finite and infinitesimal rigidity with polyhedral norms}

\author[D. Kitson]{D. Kitson}
\thanks{Supported by EPSRC grant  EP/J008648/1.}
\address{Dept.\ Math.\ Stats.\\ Lancaster University\\
Lancaster LA1 4YF \\U.K. }
\email{d.kitson@lancaster.ac.uk}

\subjclass[2010]{52C25, 52A21, 52B12}
\keywords{Bar-joint framework, infinitesimally rigid, Laman's theorem, polyhedral norm}

\begin{abstract}
We characterise finite and infinitesimal rigidity for bar-joint frameworks in $\bR^d$ with respect to polyhedral norms (i.e. norms with closed unit ball $\P$ a convex $d$-dimensional polytope). 
Infinitesimal and continuous rigidity are shown to be equivalent for finite frameworks in $\bR^d$ which are well-positioned with respect to $\P$.
An edge-labelling determined by the facets of the unit ball and placement of the framework is used to characterise infinitesimal rigidity in $\bR^d$ in terms of monochrome spanning trees.
An analogue of Laman's theorem is obtained  for all polyhedral norms on $\bR^2$.

\end{abstract}

\maketitle
\tableofcontents

\section*{Introduction}
A bar-joint framework in $\bR^d$ is a pair $(G,p)$ consisting of a simple undirected graph $G=(V(G),E(G))$ (i.e. no loops or multiple edges) and a placement $p:V(G)\to \bR^d$ of the vertices such that $p_v$ and $p_w$ are distinct whenever $vw$ is an edge of $G$.
Given a norm on $\bR^d$ we are interested in determining when a given framework can be continuously and nontrivially deformed without altering the lengths of the bars.
A well-developed rigidity theory exists in the Euclidean setting for finite bar-joint frameworks (and their variants) which stems from classical results of A. Cauchy \cite{cau}, J. C. Maxwell \cite{max},  A. D. Alexandrov  \cite{alex} and G. Laman \cite{Lam}.  
Of particular relevance is Laman's landmark characterisation for generic minimally infinitesimally rigid finite bar-joint frameworks in the Euclidean plane. Asimow and Roth proved the equivalence of finite and infinitesimal rigidity for regular bar-joint frameworks in two key papers \cite{asi-rot}, \cite{asi-rot-2}. A modern treatment can be found in works of Graver, Servatius and Servatius \cite{gra-ser-ser} and Whiteley \cite{whi-1}, \cite{whi-2}. More recently,  significant  progress has been made in topics such as global rigidity (\cite{con-1}, \cite{gor-hea-thu},  \cite{jac-jor}) and the rigidity of periodic frameworks (\cite{bor-str}, \cite{mal-the}, \cite{pow-poly}, \cite{ros-kavli})  in addition to newly emerging themes such as symmetric frameworks \cite{schulze} and frameworks supported on surfaces \cite{NOP}. 
In this article we consider rigidity properties of both finite and infinite bar-joint frameworks $(G,p)$ in $\mathbb{R}^d$ with respect to polyhedral norms. A norm on $\bR^d$ is polyhedral (or a block norm) if the closed unit ball $\{x\in \bR^d:\|x\|\leq1\}$ is the convex hull of a finite set of points.
Such norms are important from a number of perspectives.
Firstly, every norm on $\bR^d$ may be approximated by a polyhedral norm. Secondly, 
polyhedral norms are used in diverse areas of mathematical modelling. 
Thirdly, the rigidity theory obtained with polyhedral norms is distinctly different to the Euclidean setting in admitting edge-labelling and spanning tree methods. A study of rigidity with respect to
the classical non-Euclidean $\ell^p$ norms was initiated
in \cite{kit-pow}  for finite bar-joint frameworks and further developed for infinite bar-joint frameworks in  \cite{kit-pow-2}. Among these norms the $\ell^1$ and $\ell^\infty$ norms are simple examples of polyhedral norms and so the results obtained here extend some of the results of \cite{kit-pow}. 

In Section \ref{Preliminaries} we provide the relevant background material on polyhedral norms and finite and infinitesimal rigidity.
In Section \ref{InfFlexes} we establish the role of support functionals in determining the space of infinitesimal flexes of a bar-joint framework (Theorem \ref{flex1}). We then distinguish between general bar-joint frameworks and those which are well-positioned with respect to the unit ball. 
The well-positioned placements of a finite graph are open and dense in the set of all placements and we show that  finite and infinitesimal rigidity are equivalent for these bar-joint frameworks (Theorem \ref{t:rigiditythm2}).  We then introduce the rigidity matrix for a general finite bar-joint framework, the non-zero entries of which are derived from extreme points of the polar set of the unit ball.
In Section \ref{Edge} we apply an edge-labelling to $G$ which is induced by the placement of each bar in $\bR^d$ relative to the facets of the unit ball. With this edge-labelling we  identify necessary conditions for infinitesimal rigidity and obtain a sufficient condition for a subframework to be relatively infinitesimally rigid (Proposition \ref{RelRigid}). We then characterise the infinitesimally rigid bar-joint frameworks with $d$ induced framework colours as those which contain monochrome spanning trees of each framework colour (Theorem \ref{RigThm}).
This result holds for both finite and infinite bar-joint frameworks and does not require the framework to be well-positioned. 
For minimal infinitesimal rigidity we must assume that the bar-joint framework is well-positioned and an example is provided to demonstrate this.
In Section \ref{Sparsity} we apply the spanning tree characterisation to show that certain graph moves preserve minimal infinitesimal rigidity for any polyhedral norm on $\bR^2$. 
We then show that in two dimensions a finite graph has a well-positioned  minimally infinitesimally rigid placement if and only if it satisfies the counting conditions $|E(G)|=2|V(G)|-2$ and $|E(H)|\leq 2|V(H)|-2$ for all subgraphs $H$ (Theorem \ref{Laman}).  This is an analogue of  Laman's theorem \cite{Lam} which characterises the finite graphs with minimally infinitesimally rigid generic placements in the  Euclidean plane as those which satisfy the counting conditions $|E(G)|=2|V(G)|-3$ and $|E(H)|\leq 2|V(H)|-3$ for subgraphs $H$ with at least two vertices.
Many of the results obtained hold equally well for both finite and infinite bar-joint frameworks. We conclude in Section \ref{Infinite} with a discussion of some aspects which are unique to the infinite case. Illustrative examples are provided throughout.


\section{Preliminaries}
\label{Preliminaries}
Let $\P$ be a convex symmetric $d$-dimensional  polytope in $\bR^d$ where $d\geq 2$.
Following \cite{grun} we say that a proper face of $\P$ is a subset of the form 
$\P\cap H$ where $H$ is a supporting hyperplane for $\P$.
A facet of $\P$ is a proper face which is maximal with respect to inclusion.
The set of extreme points (vertices) of $\P$ is denote $ext(\P)$.
The polar set of $\P$ is denoted $\P^\triangle$ and is also a convex symmetric $d$-dimensional polytope in $\bR^d$, 
\begin{eqnarray}
\label{PolarSet}
\P^\triangle = \{y\in\bR^d: x\cdot y\leq 1,\,\,\forall\,\, x\in\P\}
\end{eqnarray}
Moreover, there exists a bijective map which assigns to each facet $F$ of $\P$ a unique extreme point  $\hat{F}$ of $\P^\triangle$ such that
\begin{eqnarray}
\label{Facet}
F=\{x\in \P: x \cdot \hat{F}=1\}
\end{eqnarray}
The polar set of $\P^\triangle$ is $\P$. 

The Minkowski functional (or gauge) for $\P$ defines a norm on $\bR^d$,
\begin{eqnarray*}
\label{norm1}
\|x\|_\P = \inf \{\lambda\geq 0:x\in \lambda \P\}
\end{eqnarray*}
This is what is known as a polyhedral norm or a block norm. 
The dual norm of $\|\cdot\|_\P$ is also a polyhedral norm and is determined by the polar set $\P^\triangle$,
\[\|y\|_\P^\ast = \max_{x\in \P} \, x\cdot y = \inf \{\lambda\geq 0:y\in \lambda \P^\triangle\}  =\|y\|_{\P^\triangle}\]
In general, a linear functional on a convex polytope will achieve its maximum value at some extreme point of the polytope and so the polyhedral norm $\|\cdot\|_\P$ is characterised by,
\begin{eqnarray}
\label{norm2}
\|x\|_\P=\|x\|_\P^{\ast\ast} =\|x\|_{\P^\triangle}^\ast= \max_{y\in \P^\triangle}\, x\cdot y
=\max_{y\in ext(\P^\triangle)}\, x\cdot y
\end{eqnarray}
A point $x\in \bR^d$ belongs to the conical hull $\cone(F)$ of a facet $F$ if
$x= \sum_{j=1}^n\lambda_jx_j$ for some non-negative scalars $\lambda_j$ and some finite set of points $x_1,x_2\ldots,x_n\in F$. 
By formulas (\ref{PolarSet}), (\ref{Facet}) and (\ref{norm2}) the following equivalence holds,
\begin{eqnarray}
\label{norm3}
x\in \cone(F) \,\,\,\,\, \Leftrightarrow \,\,\,\,\, \|x\|_\P = x\cdot \hat{F} 
\end{eqnarray}

Each isometry of the normed space $(\bR^d,\|\cdot\|_\P)$ is affine (by the Mazur-Ulam theorem) and hence is a composition of a linear isometry and a translation.
A linear isometry must leave invariant the finite set of extreme points of $\P$ and is completely determined by its action on any $d$ linearly independent extreme points. Thus there exist only finitely many linear isometries on  $(\bR^d,\|\cdot\|_\P)$.

A continuous rigid motion of $(\bR^d,\|\cdot\|_\P)$ is  a family of continuous paths, 
\[\alpha_x:(-\delta,\delta)\to\bR^d,\,\,\,\,\,\,\,\,\,x\in\bR^d\] with the property that $\alpha_x(0)=x$ and 
for every pair $x,y\in\bR^d$ the distance
$\|\alpha_x(t)-\alpha_y(t)\|_\P$ remains constant for all values of $t$.
If $\delta$ is sufficiently small then the isometries $\Gamma_t:x\mapsto \alpha_x(t)$
are necessarily translational since by continuity the linear part must equal the identity transformation.
Thus we may assume that a continuous rigid motion of  $(\bR^d,\|\cdot\|_\P)$  is a family of continuous paths of the form \[\alpha_x(t)=x+c(t), \,\,\,\,\,\,\,\,\,\,x\in\bR^d\]
for some continuous function $c:(-\delta,\delta)\to \bR^d$  (cf. \cite[Lemma 6.2]{kit-pow-2}).

An infinitesimal rigid motion of $(\bR^d,\|\cdot\|_\P)$ is a vector field on $\bR^d$ which arises from the velocity vectors of a continuous rigid motion.
Since the continuous rigid motions are initially of translational type, 
the infinitesimal rigid motions of $(\bR^d,\|\cdot\|_\P)$  are precisely the
constant maps \[\gamma:\bR^d\to\bR^d,  \,\,\,\,\,\,\, x\mapsto a\]
for some $a\in \bR^d$
(cf. \cite[Lemma 2.3]{kit-pow}).


Let $(G,p)$ be a bar-joint framework in  $(\mathbb{R}^d,\|\cdot\|_\P)$.
A {\em continuous (or finite) flex} of  $(G,p)$  is  a family of continuous paths 
\[\alpha_v:(-\delta,\delta)\to \mathbb{R}^d,\,\,\,\,\,\,\,\,\,\, v\in V(G)\] 
such that $\alpha_v(0)=p_v$ for each vertex  $v\in V(G)$ and  $\|\alpha_v(t)-\alpha_w(t)\|_\P=\|p_v-p_w\|_\P$ for all $|t|<\delta$ and each edge $vw\in E(G)$.
A continuous flex of $(G,p)$  is regarded as trivial if it arises as the restriction of a continuous rigid motion of $(\bR^d,\|\cdot\|_\P)$ to $p(V(G))$. 
If every continuous flex of $(G,p)$ is trivial then we say that $
(G,p)$ is {\em continuously rigid}.

An {\em infinitesimal flex} of $(G,p)$ is a map $u:V(G)\to \bR^d$, $v\mapsto u_v$ 
which satisfies, 
\begin{eqnarray}
\label{flexcondition}
\|(p_v+tu_v)-(p_w+tu_w)\|_\P-\|p_v-p_w\|_\P = o(t), \,\,\,\,\,\, \mbox{ as }t\to0
\end{eqnarray}
for each edge $vw\in E(G)$.
We will denote the collection of infinitesimal flexes of $(G,p)$  by $\mathcal{F}(G,p)$.
An infinitesimal flex of $(G,p)$ is regarded as trivial if it arises as the restriction of an infinitesimal rigid motion of $(\bR^d,\|\cdot\|_\P)$ to $p(V(G))$. In other words, an infinitesimal flex of $(G,p)$ is trivial if and only if it is constant. A bar-joint framework is {\em infinitesimally rigid} if every infinitesimal flex of $(G,p)$ is trivial.  Regarding $\F(G,p)$ as a real vector space with component-wise addition and scalar multiplication, the trivial infinitesimal flexes of $(G,p)$ form a $d$-dimensional subspace $\T(G,p)$ of $\F(G,p)$.
The infinitesimal flex dimension of $(G,p)$ is the vector space dimension of the quotient space $\F(G,p)/\T(G,p)$.

\section{Support functionals and rigidity}
\label{InfFlexes}
In this section we begin by highlighting the connection between the infinitesimal flex condition (\ref{flexcondition}) for a general norm on $\bR^d$ and support functionals on $(\bR^d,\|\cdot\|)$. We then characterise the space of infinitesimal flexes for a general bar-joint framework in $(\bR^d,\|\cdot\|_\P)$ in terms of support functionals and prove the equivalence of finite and infinitesimal rigidity for finite well-positioned bar-joint frameworks in $(\bR^d,\|\cdot\|_\P)$. Following this we describe the rigidity matrix for general finite bar-joint frameworks  in $(\bR^d,\|\cdot\|_\P)$ and compute some examples.

\subsection{Support functionals}
Let $\|\cdot\|$ be an arbitrary norm on $\bR^d$ and denote by $B$ the closed unit ball in $(\bR^d,\|\cdot\|)$.
A linear functional $f:\bR^d\to\bR$ is a support functional for a point $x_0\in\bR^d$ if $f(x_0)=\|x_0\|^2$ and $\|f\|^\ast=\|x_0\|$.
Equivalently, $f$ is a support functional for $x_0$ if  the hyperplane 
\[H=\{x\in\bR^d: f(x)=\|x_0\|\}\]
is a supporting hyperplane for $B$ which contains $\frac{x_0}{\|x_0\|}$.

\begin{lemma}
\label{SuppFunc}
Let $\|\cdot\|$ be a norm on $\bR^d$ and let $x_0\in\bR^d$. 
If $f:\bR^d\to\bR $ is a support functional for $x_0$ then,
\[f(y)\leq\|x_0\|\frac{\|x_0+ty\|-\|x_0\|}{t}, \,\,\,\,\,\, \forall \,\, t>0\] and 
\[f(y)\geq\|x_0\|\frac{\|x_0+ty\|-\|x_0\|}{t}, \,\,\,\,\,\, \forall\,\, t<0\]  
for all $y\in\bR^d$.
\end{lemma}

\begin{proof}
Since $f$ is linear and $f(x_0)=\|x_0\|^2$ we have for all $y\in\bR^d$,
\begin{eqnarray*}
f(y)&=& \frac{1}{t}(f(x_0+ty) - \|x_0\|^2) 
\end{eqnarray*}
If $t>0$ then since $f(x)\leq \|x_0\|\|x\|$ for all $x\in\bR^d$ we have
\begin{eqnarray*}
f(y) \leq  \|x_0\|\frac{\|x_0+ty\|-\|x_0\|}{t}
\end{eqnarray*}
If $t<0$ then applying the above inequality,
\begin{eqnarray*}
f(y) =-f(-y)\geq  -\|x_0\|\frac{\|x_0-t(-y)\|-\|x_0\|}{-t}= \|x_0\|\frac{\|x_0+ty\|-\|x_0\|}{t}
\end{eqnarray*}
\end{proof}

Let $(G,p)$ be a bar-joint framework in $(\bR^d,\|\cdot\|)$ and fix an orientation for each edge $vw\in E(G)$. We denote by $\supp(vw)$ the set of all support functionals for $p_v-p_w$.
(The choice of orientation on the edges of $G$ is for convenience only and has no bearing on the results that follow. Alternatively, we could avoid choosing an orientation by defining $\supp(vw)$ to be the set of all linear functionals which are support functionals for either $p_v-p_w$ or $p_w-p_v$.)

\begin{proposition}
\label{Supp}
If $(G,p)$ is a bar-joint framework in $(\bR^d,\|\cdot\|)$
and $u:V(G)\to\bR^d$ is an infinitesimal flex of $(G,p)$ then
\[u_v-u_w\in\bigcap_{f\in \supp(vw)}\, \ker f\]
for each edge $vw\in E(G)$.
\end{proposition}

\proof
Let $vw\in E(G)$ and suppose $f$ is a support functional for $p_v-p_w$.
Applying Lemma \ref{SuppFunc}
with $x_0=p_v-p_w$ and $y=u_v-u_w$ we have,
\begin{eqnarray*}
\lim_{t\to0^-}\frac{\|x_0+ty\|-\|x_0\|}{t}\leq\frac{f(y)}{\|x_0\|} 
\leq  \lim_{t\to 0^+}\frac{\|x_0+ty\|-\|x_0\|}{t}
\end{eqnarray*}
Since $u$ is an infinitesimal flex of $(G,p)$,
$\lim_{t\to0}\frac{1}{t}(\|x_0+ty\|-\|x_0\|)=0$
and so $f(y)=0$.

\endproof

Let $\|\cdot\|_\P$ be a polyhedral norm on $\bR^d$. For each facet $F$ of $\P$ denote by $\varphi_F$ the linear functional 
\[\varphi_F:\bR^d\to\bR, \,\,\,\,\,\,\, x\mapsto x\cdot  \hat{F}\]

\begin{lemma}
\label{SuppLemma}
Let $\|\cdot\|_\P$ be a polyhedral norm on $\bR^d$, let $F$ be a facet of $\P$ and let $x_0\in \bR^d$.
Then $x_0\in \cone(F)$ if and only if the linear functional,
\[\varphi_{F,x_0}:\bR^d\to \bR,\,\,\,\,\,\,\,\, x\mapsto \|x_0\|_\P \,\varphi_F(x)\]
is a support functional for $x_0$. 
\end{lemma}

\begin{proof}
If $x_0\in \cone(F)$ then by formula (\ref{norm3}), $\varphi_{F,x_0}\left(x_0\right)=\|x_0\|_\P^2$.
By (\ref{PolarSet}) we have $\varphi_{F,x_0}(x)\leq \|x_0\|_\P$ for each $x\in \P$ and
it follows that $\varphi_{F,x_0}$ is a support functional for $x_0$. Conversely, if $x_0\notin \cone(F)$ then by (\ref{norm3}), $\varphi_{F,x_0}(x_0)<\|x_0\|_\P^2$ and so $\varphi_{F,x_0}$ is not a support functional for $x_0$.
\end{proof}

For each oriented  edge $vw\in E(G)$ we denote by 
$\supp_\Phi(vw)$ the set of all linear functionals $\varphi_F$ which are support functionals for
$\frac{p_v-p_w}{\|p_v-p_w\|_\P}$.



\begin{proposition}
\label{finiteflex1}
Let $(G,p)$ be a finite bar-joint framework in $(\mathbb{R}^d,\|\cdot\|_\P)$.
If   $u:V(G)\to\bR^d$ satisfies
\[ u_v-u_w\in \bigcap_{\varphi_F\in \supp_\Phi(vw)}\, \ker \varphi_F\]
for each edge $vw\in E(G)$
then there exists $\delta>0$ such that the family 
\[\alpha_v:(-\delta,\delta)\to \bR^d,\,\,\,\,\,\,\, \alpha_v(t)= p_v+tu_v\]
is a finite flex of $(G,p)$.
\end{proposition} 

\begin{proof}
Let  $vw\in E(G)$  and write $x_0=p_v-p_w$ and $u_0=u_v-u_w$.
If $\varphi_{F}$ is a support functional for $\frac{x_0}{\|x_0\|_\P}$ then,  by the hypothesis, $\varphi_F(u_0)=0$.
By Lemma \ref{SuppLemma}, $x_0$ is contained in the conical hull of the facet $F$. 
Applying formulas (\ref{norm2}) and (\ref{norm3}),
\[\|x_0\|_\P = \max_{y\in ext(\P^\triangle)} x_0\cdot y  = x_0\cdot\hat{F}\]
By continuity there exists $\delta_{vw}>0$ such that for all $|t|<\delta_{vw}$,
\begin{eqnarray*}
\|x_0+tu_0\|_\P
&=& \max_{y\in ext(\P^\triangle)} (x_0+tu_0)\cdot y\\
&=& (x_0+tu_0)\cdot \hat{F} \\
&=& \|x_0\|_\P+t\,\varphi_F(u_0)\\
&=& \|x_0\|_\P
\end{eqnarray*}
Since $G$ is a finite graph the result holds with $\delta=\min_{vw\in E(G)}\delta_{vw}>0$.
\end{proof}

The following is a characterisation of the space of infinitesimal flexes of a general bar-joint framework in $(\bR^d,\|\cdot\|_\P)$.

\begin{theorem}
\label{flex1}
Let $(G,p)$ be a bar-joint framework in $(\mathbb{R}^d,\|\cdot\|_\P)$.
Then a mapping  $u:V(G)\to \mathbb{R}^d$ is an infinitesimal flex of $(G,p)$ if and only if 
\[ u_v-u_w\in \bigcap_{\varphi_F\in \supp_\Phi(vw)}\, \ker \varphi_F\]
for each edge $vw\in E(G)$.
\end{theorem} 

\begin{proof}
If $u$ is an infinitesimal flex of $(G,p)$ then the result follows from Proposition \ref{Supp}. 
For the converse, let  $vw\in E(G)$  and write $x_0=p_v-p_w$ and $u_0=u_v-u_w$.
Applying the argument in the proof of Proposition \ref{finiteflex1}, 
there exists $\delta_{vw}>0$ with $\|x_0+tu_0\|_\P=\|x_0\|_\P$ for all $|t|<\delta_{vw}$.
Hence $u$ is an infinitesimal flex of $(G,p)$.
\end{proof}

\subsection{Equivalence of finite and infinitesimal rigidity}
A placement of a simple graph $G$ in $\bR^d$ is a map $p:V(G)\to\bR^d$ for which $p_v\not=p_w$ whenever $vw\in E(G)$.
 A placement $p:V(G)\to\bR^d$ is {\em well-positioned} with respect to a polyhedral norm 
on $\bR^d$ if $p_v-p_w$ is contained in the conical hull of exactly one facet of the unit ball $\P$ for each edge $vw\in E(G)$. We denote this unique facet by $F_{vw}$.
In the following discussion $G$ is a finite graph and each placement  is identified with a point $p=(p_v)_{v\in V(G)}$ in the product space $\prod_{v\in V(G)}\bR^{d}$ which we regard as having the usual topology. 
 The set  of all well-positioned placements of $G$  in $(\bR^d,\|\cdot\|_\P)$ is an open and dense subset of this product space. The {\em configuration space}  for  a bar-joint framework $(G,p)$ is defined as,
\[V(G,p) = \{x\in \prod_{v\in V(G)}\bR^{d}:\|x_v-x_w\|_\P=\|p_v-p_w\|_\P, \,\, \forall \,\,vw\in 
E(G)\}\]

\begin{proposition}
\label{Configuration}
Let $(G,p)$ be a finite well-positioned bar-joint framework in $(\bR^d,\|\cdot\|_\P)$
with $p_v-p_w\in\cone(F_{vw})$ for each $vw\in E(G)$.
Then there exists a neighbourhood $U$ of $p$ in $\prod_{v\in V(G)}\bR^{d}$ 
such that,
\begin{enumerate}[(i)]
\item  if $x\in U$ then $x_v-x_w\in \cone(F_{vw})$  for each edge $vw\in E(G)$,
\item $(G,x)$ is a well-positioned bar-joint framework  for each $x\in U$, and,
\item
$V(G,p)\cap U = \{x\in U: \varphi_{F_{vw}}(x_{v}-x_{w})=\varphi_{F_{vw}}(p_{v}-p_{w}), \,\, \forall \,\,vw\in E(G)\}$.
\end{enumerate}
In particular,
 $V(G,p)\cap U= (p+\F(G,p))\cap U$.
\end{proposition}

\begin{proof}
Let $vw\in E(G)$ be an oriented edge and consider the continuous map, 
\[T_{vw}:\prod_{v'\in V(G)}\bR^{d}\to\bR,
\,\,\,\,\,\, (x_{v'})_{v'\in V(G)}\mapsto x_{v}-x_{w}\]
Since $(G,p)$ is well-positioned, $p_v-p_{w}$ is an interior point of the conical hull of a unique facet 
$F_{vw}$ of $\P$. The preimage $T_{vw}^{-1}(\cone(F_{vw})^\circ)$ is an open neighbourhood of $p$. Since  $G$ is a finite graph the intersection, 
\[U = \bigcap_{vw\in E(G)} T_{vw}^{-1}(\cone(F_{vw})^\circ)\]
is an open neighbourhood of $p$ which satisfies $(i)$, $(ii)$ and $(iii)$. 

Since $(G,p)$ is well-positioned, by Lemma \ref{SuppLemma}, there is exactly one support functional in $\supp_\Phi(vw)$ for each edge $vw$ and this functional is given by $\varphi_{F_{vw}}$.
If $x\in  U$ then  define $u=(u_v)_{v\in V(G)}$ by setting 
$u_v=x_v-p_v$ for each $v\in V(G)$.
By $(iii)$, $x\in V(G,p)\cap U$ if and only if $x\in U$ and 
\[ \varphi_{F_{vw}} (u_v-u_w)= \varphi_{F_{vw}}(x_v-x_w)- \varphi_{F_{vw}}(p_v-p_w)=0\] for each edge $vw\in E(G)$.
By Theorem \ref{flex1}, the latter identity is equivalent to the condition that $u$ is an infinitesimal flex of $(G,p)$.
Thus $x\in V(G,p)\cap U$ if and only if $x\in U$ and $x-p\in \F(G,p)$.
\end{proof}

We now prove the equivalence of continuous rigidity and infinitesimal rigidity for finite well-positioned bar-joint frameworks. 

\begin{theorem}\label{t:rigiditythm2}
Let $(G,p)$ be a finite well-positioned   bar-joint framework in $(\bR^d,\|\cdot\|_\P)$.
Then the following statements are equivalent.
\begin{enumerate}[(i)]
\item
$(G,p)$ is continuously rigid.
\item
$(G,p)$ is infinitesimally rigid.
\end{enumerate}
\end{theorem}

\begin{proof}
$(i)\Rightarrow (ii)$.
If $u=(u_v)_{v\in V(G)}\in \F(G,p)$ is an infinitesimal flex of $(G,p)$ then by Theorem \ref{flex1} and Proposition \ref{finiteflex1}, the family 
\[\alpha_v:(-\epsilon,\epsilon)\to\bR^d,\,\,\,\,\,\,\, \alpha_v(t)=p_v+tu_v,\,\,\,\,\,\, v\in V(G)\]
is a finite flex of $(G,p)$ for some $\epsilon>0$.
Since $(G,p)$ is continuously rigid this finite flex must be trivial.
Thus there exists  $\delta>0$ and a continuous path 
$c:(-\delta,\delta)\to \bR^d$ such that $\alpha_v(t)=p_v+c(t)$ for all $|t|<\delta$ and all $v\in V(G)$.
Now $u_v=\alpha_v'(0)=c'(0)$ for all $v\in V(G)$ and so $u$ is a constant, and hence trivial, infinitesimal flex of $(G,p)$.
We conclude that $(G,p)$ is infinitesimally rigid.

$(ii)\Rightarrow (i)$.
If $(G,p)$ has a finite flex given by the family,
\[\alpha_v:(-\epsilon,\epsilon)\to\bR^d,\,\,\,\,\,\, v\in V(G)\]
then consider the continuous path,
\[\alpha:(\epsilon,\epsilon)\to V(G,p), \,\,\,\,\,\,\,\, t\mapsto (\alpha_v(t))_{v\in V(G)}\]
By Proposition \ref{Configuration}, $V(G,p)\cap U=(p+\F(G,p))\cap U$ for some neighbourhood $U$ of $p$. Since $\alpha(0)=p$,  there exists $\delta>0$ such that
$\alpha(t)\in V(G,p)\cap U$ for all $|t|<\delta$.
Choose $t_0\in(-\delta,\delta)$ and define, 
\[u:V(G)\to\bR^d, \,\,\,\,\,\,\, u_v=\alpha_v(t_0)-p_v\]
Then $u=\alpha(t_0)-p\in\F(G,p)$ is an infinitsimal flex of $(G,p)$.
Since $(G,p)$ is infinitesimally rigid, $u$ must be a trivial infinitesimal flex.
Hence $u_v=c(t_0)$ for all $v\in V(G)$ and some $c(t_0)\in\bR^d$.
Apply this same argument  to show that for each $|t|<\delta$ there exists $c(t)$ such that 
$\alpha_v(t) = p_v+c(t)$ for all $v\in V(G)$.
Note that $c:(-\delta, \delta)\to \bR^d$ is continuous
and so $\{\alpha_v:v\in V(G)\}$ is a trivial finite flex of $(G,p)$.
We conclude that $(G,p)$ is continuously rigid.
\end{proof}

The non-equivalence of finite and infinitesimal rigidity for general finite bar-joint frameworks in $(\bR^d,\|\cdot\|_\P)$ is demonstrated in Examples \ref{K2Example} and \ref{RayEx}.




\subsection{The rigidity matrix}
We define the {\em rigidity matrix} $R_\P(G,p)$ for a finite
bar-joint framework $(G,p)$  in $(\bR^d,\|\cdot\|_\P)$ as follows:
Fix an ordering of the vertices $V(G)$ and edges $E(G)$ and choose an orientation 
on the edges of $G$.
For each vertex $v$ assign $d$ columns in the rigidity matrix and label these columns
$p_{v,1},\ldots,p_{v,d}$.
For each directed edge $vw\in E(G)$ and each facet $F$ with $p_v-p_w\in \cone(F)$
assign a row in the rigidity matrix
and label this row by $(vw,F)$.
The entries for the row $(vw,F)$ are given by 
\small{ \begin{eqnarray*}
\kbordermatrix{
& && & p_{v,1}\,&\cdots &\,p_{v,d} \,& && & p_{w,1}& \cdots &\, \,p_{w,d}& & &  \\
&0\,\, &\cdots& \,0& \,\hat{F}_1\,\,&  \cdots & \,\hat{F}_d\, & 0\,\,& \cdots&\, 0 &\,-\hat{F}_1\,\,& \cdots & \,-\hat{F}_d\,& 0\,\,& \cdots&\, 0}
\end{eqnarray*}}
where $p_v-p_w\in cone(F)$ and 
$\hat{F}=(\hat{F}_1,\ldots,\hat{F}_d)\in \bR^d$.
If $(G,p)$ is well-positioned then the rigidity matrix has size $|E(G)|\times d|V(G)|$.
 

\begin{proposition}
\label{RigidityMatrixProp}
Let $(G,p)$ be a finite bar-joint framework in $(\mathbb{R}^d,\|\cdot\|_\P)$.
Then \begin{enumerate}[(i)]
\item
$\F(G,p)\cong\ker R_\P(G,p)$.
\item 
$(G,p)$ is infinitesimally rigid if and only if $\rank R_\P(G,p)= d|V(G)|-d$.
\end{enumerate}
\end{proposition}

\begin{proof}
The system of equations in Theorem \ref{flex1} is expressed by the matrix equation $R_\P(G,p)u^T=0$ where we  identify $u:V(G)\to\bR^d$ with a row vector $(u_{v_1},\ldots,u_{v_n})\in\bR^{d|V(G)|}$.
Thus $\F(G,p)\cong\ker R_\P(G,p)$.
The space of trivial infinitesimal flexes of $(G,p)$ has dimension $d$ and so 
in general we have \[ \rank R_\P(G,p)\leq d|V(G)|-d\] with equality if and only if
 $(G,p)$ is infinitesimally rigid.
\end{proof}


If $F$ is a facet of $\P$  and  $y_1,y_2,\ldots,y_{d}\in ext(\P)$ are extreme points of $\P$ which are contained in $F$ then
for each column vector $y_k$ we compute 
$[1 \cdots 1]\,A^{-1}\, y_k=1$ 
where 
$A=[y_1\cdots y_d]\in M^{d\times d}(\bR)$.
Hence 
\begin{eqnarray}
\label{RigMatrixEntries1}
\hat{F} &=& [1 \cdots 1]A^{-1}
\end{eqnarray}
Moreover, if $y_1,y_2,\ldots,y_{d}$ are pairwise orthogonal then $A^{-1}=\left[\frac{y_1}{\|y_1\|^2_2} \cdots \frac{y_d}{\|y_d\|^2_2}\right]^T$ and so
\begin{eqnarray}
\label{RigMatrixEntries}
\hat{F} =\sum_{j=1}^d \frac{y_j}{\|y_j\|^2_2}
\end{eqnarray}
where $\|\cdot\|_2$ is the Euclidean norm on $\bR^d$.


\begin{example}
\label{K2Example}
Let $\P$ be a crosspolytope in $\bR^d$ with $2d$ many extreme points $ext(\P)=\{\pm e_k:k=1,\ldots,d\}$ where $e_1,e_2,\ldots,e_d$ is the usual basis in $\bR^d$.
Then each facet $F$ contains  $d$ pairwise orthogonal extreme points $y_1,y_2,\ldots,y_d$ 
each of Euclidean norm $1$. By (\ref{RigMatrixEntries}), $\hat{F} = \sum_{j=1}^d y_j$ and the resulting polyhedral norm is the $1$-norm
\[\|x\|_\P=\max_{y\in ext(\P^\triangle)} x\cdot y = \sum_{i=1}^d|x_i| = \|x\|_1\]
Consider for example the placements of the complete graph $K_2$ 
in $(\bR^2,\|\cdot\|_1)$ illustrated in Figure \ref{1NormFlex}. 
The polytope $\P$ is indicated on the left with facets labelled $F_1$ and $F_2$.
The extreme points of the polar set $\P^\triangle$ which correspond to these facets are $\hat{F}_1=e_1+e_2=(1,1)$ and $\hat{F}_2=e_1-e_2=(1,-1)$.
The first placement is well-positioned with respect to $\P$ and the rigidity matrix is,
\[ \kbordermatrix{
& p_{v,1} & p_{v,2} & & p_{w,1} & p_{w,2} \\
(vw,F_1)& 1 & 1 & \vrule & -1 & -1 }\]
This bar-joint framework has infinitesimal flex dimension $1$.
The second placement is not well-positioned and the rigidity matrix is,
\[ \kbordermatrix{
& p_{v,1} & p_{v,2} & & p_{w,1} & p_{w,2} \\
(vw,F_1)&1 & 1 & \vrule & -1 & -1 \\
(vw,F_2)& 1 & -1 & \vrule & -1 & 1} \]
As the rigidity matrix has rank $2$ this bar-joint framework  
is infinitesimally rigid in $(\mathbb{R}^2,\|\cdot\|_1)$, but continuously flexible. 
\end{example}

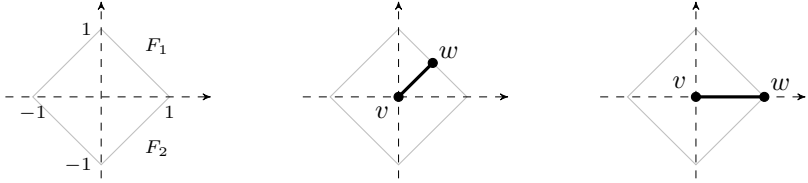
\begin{figure}[h]
\centering
  \begin{tabular}{  c   }
  
    \begin{minipage}{.3\textwidth}
    \begin{tikzpicture}[scale=0.9, axis/.style={very thin,dashed, ->, >=stealth'},
    important line/.style={very thick},
    dashed line/.style={dashed,  thick},
    pile/.style={thick, ->, >=stealth', shorten <=2pt, shorten
    >=2pt},
    every node/.style={color=black}]
 
  \clip (-2,-1.3) rectangle (2cm,1.5cm); 
  
  \coordinate (A1) at (-1,0);
  \coordinate (A2) at (0,0);
  \coordinate (A3) at (1,0);
  \coordinate (A4) at (0,1);
  \coordinate (A5) at (0,-1); 

  \draw[lightgray] (A3) -- (A4) -- (A1) -- (A5) -- cycle;
 
   \draw[axis] (-1.4,0)  -- (1.6,0) node(xline)[right]
        {$ $};
   \draw[axis] (0,-1.2) -- (0,1.4) node(yline)[above] {$ $};
	




   \node[above right] at (0.5,0.5)  {\tiny $F_1$};
  \node[below right] at (0.5,-0.5)  {\tiny $F_2$};	

\node[below] at (A1)  {\tiny $-1$};
  \node[below] at (A3)  {\tiny $1$};	
  \node[left] at (A4)  {\tiny $1$};	
  \node[left] at (A5)  {\tiny $-1$};	
\end{tikzpicture}

\end{minipage}

  \begin{minipage}{.3\textwidth}
    \begin{tikzpicture}[scale=0.9, axis/.style={very thin,dashed, ->, >=stealth'},
    important line/.style={very thick},
    dashed line/.style={dashed,  thick},
    pile/.style={thick, ->, >=stealth', shorten <=2pt, shorten
    >=2pt},
    every node/.style={color=black}]
 
  \clip (-2,-1.3) rectangle (2cm,1.5cm); 
  
  \coordinate (A1) at (-1,0);
  \coordinate (A2) at (0,0);
  \coordinate (A3) at (1,0);
  \coordinate (A4) at (0,1);
  \coordinate (A5) at (0,-1); 

 \draw[lightgray] (A3) -- (A4) -- (A1) -- (A5) -- cycle;
 
   \draw[axis] (-1.4,0)  -- (1.6,0) node(xline)[right]
        {$ $};
   \draw[axis] (0,-1.2) -- (0,1.4) node(yline)[above] {$ $};
	
   \draw[important line] (0,0) coordinate (A) -- (0.5,0.5)
        coordinate (B) node[right, text width=5em] {$  $};


   \node[draw,circle,inner sep=1.2pt,fill] at (A2) {};
   \node[draw,circle,inner sep=1.2pt,fill] at (0.5,0.5) {};

  \node[right] at (0.45,0.65)  {\small $w$};	
  \node[below left] at (0,0)  {\small $v$};	

\end{tikzpicture}

\end{minipage}

\begin{minipage}{.3\textwidth}
    \begin{tikzpicture}[scale=0.9, axis/.style={very thin,dashed, ->, >=stealth'},
    important line/.style={very thick},
    dashed line/.style={dashed,  thick},
    pile/.style={thick, ->, >=stealth', shorten <=2pt, shorten
    >=2pt},
    every node/.style={color=black}]
 
  \clip (-2,-1.3) rectangle (2cm,1.5cm); 
  
  \coordinate (A1) at (-1,0);
  \coordinate (A2) at (0,0);
  \coordinate (A3) at (1,0);
  \coordinate (A4) at (0,1);
  \coordinate (A5) at (0,-1); 

  \draw[lightgray] (A3) -- (A4) -- (A1) -- (A5) -- cycle;

  \draw[axis] (-1.4,0)  -- (1.6,0) node(xline)[right]
        {$ $};
  \draw[axis] (0,-1.2) -- (0,1.4) node(yline)[above] {$ $};
	
  \draw[important line] (0,0) coordinate (A) -- (1,0)
        coordinate (B) node[right, text width=5em] {$ $};


  \node[draw,circle,inner sep=1.2pt,fill] at (A2) {};
  \node[draw,circle,inner sep=1.2pt,fill] at (A3) {};

  \node[right] at (0.95,0.19)  {\small $w$};	
  \node[above left] at (0,0)  {\small $v$};	


\end{tikzpicture}

\end{minipage}
\end{tabular}
  \caption{An infinitesimally flexible  and  an infinitesimally rigid placement  of $K_2$ in $(\mathbb{R}^2,\|\cdot\|_1)$.}
\label{1NormFlex}
\end{figure}

\section{Edge-labellings and monochrome subgraphs}
\label{Edge}
In this section we describe an edge-labelling on $G$ which depends on the placement of the bar-joint framework $(G,p)$  in $(\mathbb{R}^d,\|\cdot\|_\mathcal{P})$ relative to the facets of $\P$. 
We provide methods for identifying infinitesimally flexible frameworks and subframeworks which are relatively infinitesimally rigid. We then characterise infinitesimal rigidity for bar-joint frameworks with $d$ framework colours  in terms of the monochrome subgraphs induced by this edge-labelling.  

\subsection{Edge-labellings}
Let $(G,p)$ be a general bar-joint framework in $(\bR^d,\|\cdot\|_\P)$.
Since $\P$ is symmetric in $\bR^d$, if $F$ is a facet of $\P$ then $-F$ is also a facet of $\P$.
Denote by $\Phi(\P)$ the collection of all pairs $[F]=\{F,-F\}$.
For each edge $vw\in E(G)$ define
\[\Phi(vw)=\{[F]\in\Phi(\P):p_v-p_w\in \cone(F)\cup \cone(-F)\}\]
We refer to the elements of $\Phi(vw)$ as the {\em framework colours} of the edge $vw$.
For example, if $p_v-p_w$ lies in the conical hull of exactly one facet of $\P$ then the edge $vw$ has just one framework colour.
If $p_v-p_w$ lies along a ray through an extreme point of $\P$ then $vw$ has at least $d$ distinct framework colours. 
By Lemma \ref{SuppLemma}, $[F]$ is a framework colour for an edge $vw$ if and only if either $\varphi_F$ or $-\varphi_F$ is a support functional for $\frac{p_v-p_w}{\|p_v-p_w\|_\P}$.

For each vertex $v_0\in V(G)$ denote by $\Phi(v_0)$ the collection of framework colours of all edges which are incident with $v_0$,  
\[\Phi(v_0)=\bigcup_{v_0w\in E(G)} \Phi(v_0w)\]


\begin{proposition}
\label{Vertex}
 If $(G,p)$ is an infinitesimally rigid bar-joint framework 
in $(\mathbb{R}^d,\|\cdot\|_\mathcal{P})$ then $|\Phi(v)|\geq d$
for each vertex $v\in V(G)$.
\end{proposition}

\begin{proof}
 If $v_0\in V(G)$ and $|\Phi(v_0)|< d$ then there exists non-zero 
\[x\in \bigcap_{[F]\in \Phi(v_0)} \ker \varphi_{F}\]
By Theorem \ref{flex1}, if $u:V(G)\to\bR^d$ is defined by
\[u_v = \left\{ \begin{array}{ll}
x & \mbox{ if }v=v_0 \\
0 & \mbox{ if }v\not=v_0
\end{array}\right. \]
then $u$ is a non-trivial infinitesimal flex of $(G,p)$.
\end{proof}

We now consider the subgraphs of $G$ which are spanned by edges possessing a particular framework colour. For each facet $F$ of $\P$ define
\[E_F(G,p)=\{vw\in E(G):[F]\in\Phi(vw)\}\]
and let $G_F$ be the subgraph of $G$ spanned by $E_F(G,p)$.
We refer to  $G_F$ as a {\em monochrome subgraph} of $G$.

\begin{example}
Let $\P$ be a hypercube in $\bR^d$ with $2^d$ many extreme points
$ext(\P)=\{\sum_{k=1}^d (-1)^{i_k}e_{k}:i_1,\ldots,i_d\in \{0,1\}\}$. 
Then each facet $F$ contains $2^{d-1}$ extreme points of $\P$ each of Euclidean norm $\sqrt{d}$.
Among these extreme points there exist $d$ which are pairwise orthogonal $y_1,y_2,\ldots,y_d$. Thus by (\ref{RigMatrixEntries}), $\hat{F} = \frac{1}{d}(\sum_{j=1}^d y_j)=\pm e_k$ for some $k$.
The resulting polyhedral norm is the maximum norm,
\[\|x\|_\P=\max_{y\in ext(\P^\triangle)} x\cdot y=\max_{k=1,2,\ldots,d}  |x_i| =  \|x\|_\infty\]
For example, consider the placement $p$ of the complete graph $K_3$ in $(\bR^2,\|\cdot\|_\infty)$ illustrated in Figure \ref{K3Fig}. The polytope $\P$ is  indicated on the left with facets labelled $F_1$ and $F_2$.
This bar-joint framework is well-positioned with respect to $\P$ as each edge has exactly one framework colour, 
\[\Phi(ab) = [F_1],\,\,\,\, \,\,\,\,\,\Phi(ac) = [F_2],\,\, \,\,\,\,\,\,\,
\Phi(bc) = [F_2]\]
The monochrome subgraphs $G_{F_1}$ and  $G_{F_2}$ are indicated in black and gray respectively. The corresponding extreme points of $\P^\triangle$ are
$\hat{F}_1=(1,0)$ and $\hat{F}_2=(0,1)$.
The rigidity matrix has rank $3$ and so $(K_3,p)$ has infinitesimal flex dimension $1$. 
The edges which are incident with the vertex $c$ each have framework colour $[F_2]$ and so a non-trivial infinitesimal flex of $(K_3,p)$ may be obtained as in the proof of Proposition \ref{Vertex}. 
\end{example}

\begin{figure}[ht]
  \centering
 \begin{tabular}{  c   }
 \begin{minipage}{.22\textwidth}
\begin{tikzpicture}[scale=0.65, axis/.style={very thin,dashed, ->, >=stealth'},
    important line/.style={very thick},
    dashed line/.style={dashed,  thick},
    pile/.style={thick, ->, >=stealth', shorten <=2pt, shorten
    >=2pt},
    every node/.style={color=black}]
 
  \clip (-2,-1.4) rectangle (2cm,1.6cm); 
  
  \coordinate (A1) at (-1,1);
  \coordinate (A2) at (0,0);
  \coordinate (A3) at (1,-1);
  \coordinate (A4) at (1,1);
  \coordinate (A5) at (-1,-1); 

  \draw[lightgray] (A3) -- (A4) -- (A1) -- (A5) -- cycle;
    
  \draw[axis] (-1.7,0)  -- (1.8,0) node(xline)[right]
        {$ $};
  \draw[axis] (0,-1.3) -- (0,1.5) node(yline)[above] {$ $};

  \draw[dashed] (0,-1.3) -- (0,1.5);

  \node[below left] at (-1,0)  {\tiny $-1$};
  \node[below right] at (1,0)  {\tiny $1$};	

  \node[above left] at (0,1)  {\tiny $1$};	
  \node[below left] at (0,-1)  {\tiny $-1$};	

  \node[above right] at (0.2,1)  {\tiny $F_2$};	
  \node[above right] at (1,0.1)  {\tiny $F_1$};	
\end{tikzpicture}
\end{minipage}

\begin{minipage}{.32\textwidth}
  \begin{tikzpicture}[scale=0.55]
 
    \clip (-3.6,-0.38) rectangle (3.2cm,2.9cm); 
  
  \coordinate (Aone) at (-1,0);
  \coordinate (Athree) at (1,0);
  \coordinate (Afour) at (0,2);

 \draw[thick] (Aone)  -- (Athree);
 \draw[thick,lightgray] (Athree)  -- (Afour);
 \draw[thick,lightgray] (Aone)  --  (Afour);	

  \node[draw,circle,inner sep=1.4pt,fill] at (Aone) {};

  \node[draw,circle,inner sep=1.4pt,fill] at (Athree) {};
  \node[draw,circle,inner sep=1.4pt,fill] at (Afour) {};

  \node[left] at (Aone) {\small $a(-1,0)$};
  \node[left] at (Afour) {\small $c(0,2)$};
  \node[right] at (Athree) {\small $b(1,0)$};
  \end{tikzpicture}
\end{minipage}

\begin{minipage}{.40\textwidth}
{\small $ \kbordermatrix{
 & a_x&   a_y\,\,& b_x & b_y   \,\,&  c_x & c_y \\
ab& 1  & 0\,\, \vrule  &-1 & 0 \,\, \vrule & 0 & 0 \\
bc& 0 & 0 \,\, \vrule & 0 & 1 \,\, \vrule &  0 & -1\\
ac& 0 & 1 \,\, \vrule & 0 & 0 \,\, \vrule & 0 & -1
}$}
\end{minipage}
\end{tabular}

 \caption{A placement and rigidity matrix for $K_3$ in $(\bR^2,\|\cdot\|_\infty)$}
 \label{K3Fig}
\end{figure}
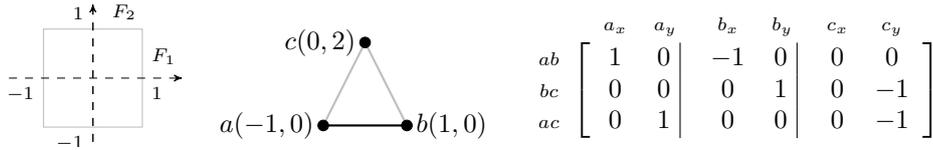




Denote by $\Phi(G,p)$ the collection of all framework colours of edges of $G$,
\[\Phi(G,p)=\bigcup_{vw\in E(G)} \Phi(vw)\]
We refer to the elements of $\Phi(G,p)$ as the {\em framework colours} of the bar-joint framework.
 $(G,p)$
\begin{proposition}
\label{Polytope2}
 Let $(G,p)$ be an  infinitesimally rigid bar-joint framework in $(\mathbb{R}^d,\|\cdot\|_\mathcal{P})$.
 If $C$ is a collection of framework colours of $(G,p)$ with 
$|\Phi(G,p)\backslash C|<d$  then \[\bigcup_{[F]\in C}G_F\]
  contains a spanning tree of $G$.
\end{proposition}

\begin{proof}
Suppose that $\bigcup_{[F]\in C}G_F$ does not contain a spanning tree of $G$.
Then there exists a partition $V(G) = V_1 \cup V_2$ for which there is
no edge $v_1v_2\in E(G)$ with framework colour contained in $C$ satisfying $v_1\in V_1$ and  $v_2\in V_2$. 
Since $|\Phi(G,p)\backslash C|<d$ there exists non-zero
\[x\in\bigcap_{[F]\in\Phi(G,p)\backslash C} \ker \varphi_{F}\] 
By Theorem \ref{flex1}, if $u:V(G)\to\bR^d$ is defined by
\[u_v = \left\{ \begin{array}{ll}
x & \mbox{ if }v\in V_1 \\
0 & \mbox{ if }v\in V_2
\end{array}\right. \]
then $u$ is a non-trivial infinitesimal flex of $(G,p)$.
We conclude that $\bigcup_{[F]\in C}G_F$ contains a spanning tree of $G$.

\end{proof}

The converse  to Proposition \ref{Polytope2} does not hold in general as the following example illustrates. In Theorem \ref{RigThm} we show that a converse statement
does  hold under  the additional assumption that $|\Phi(G,p)|=d$.

\begin{example}
A norm on $\bR^d$ is {\em additive} if there exists a finite set $B\subset \bR^d$ such that
\[\|x\|_\P=\sum_{b\in B} |x\cdot b|, \,\,\,\,\,\, \forall \,\, x\in\bR^d\]
Every  norm of this type is a polyhedral norm. 
If $F$ is a facet of the closed unit ball and  $x$ is an interior point of the conical hull of $F$ then
\[\hat{F} = \sum_{b\in B}\sgn(x\cdot b)b\]
Consider the polyhedral norm on $\bR^2$ given by $\|x\|_\P=|x\cdot b_1|+|x\cdot b_2| + |x\cdot b_3|$ where $b_1=(1,0)$, $b_2=(0,1)$ and $b_3=(1,1)$.
Let $(K_3,p)$ be the bar-joint framework in $(\bR^2,\|\cdot\|_\P)$ which is illustrated in Figure \ref{Additive}. The monochrome subgraphs corresponding to the facets $F_1$, $F_2$ and $F_3$ 
are indicated by black, gray and dashed lines respectively.
The rigidity matrix is
\[R_\P(G,p)= \kbordermatrix{
&& a_x && a_y &&&& b_x && b_y &&&  &  c_x && c_y& \\
(ab,F_1)&& -2  && -2 & &\vrule& & 2& & 2 & &\vrule& & 0 && 0& \\
(bc,F_3)&& 0 && 0 && \vrule && 2 && 0 && \vrule& &  -2 && 0&\\
(ac,F_2)& &0 && -2 && \vrule && 0 && 0 && \vrule && 0 && 2&
}\]
Note that if $C$ is any collection of at least two framework colours then  $|\Phi(G,p)\backslash C|<2$ and $\bigcup_{[F]\in C}\,G_F$ contains a spanning tree of $G$.
However, the rigidity matrix has rank $3$ and so the infinitesimal flex dimension of $(K_3,p)$ is $1$.
\end{example}

 \begin{figure}[ht]
\centering
  \begin{tabular}{  c   }
  
 \begin{minipage}{.45\textwidth}
\begin{tikzpicture}[scale=0.9, axis/.style={very thin,dashed, ->, >=stealth'},
    important line/.style={very thick},
    dashed line/.style={dashed,  thick},
    pile/.style={thick, ->, >=stealth', shorten <=2pt, shorten
    >=2pt},
    every node/.style={color=black}]
 
  \clip (-3.5,-1.5) rectangle (2cm,1.5cm); 
  
  \coordinate (A1) at (1,0);
  \coordinate (A2) at (0,1);
  \coordinate (A3) at (-1,1);
  \coordinate (A4) at (-1,0);
  \coordinate (A5) at (0,-1); 
  \coordinate (A6) at (1,-1);
  \coordinate (A7) at (1,0); 

  \draw[lightgray] (A1) -- (A2) -- (A3) -- (A4) -- (A5) -- (A6) -- (A7) -- cycle;
    
  \draw[axis] (-1.7,0)  -- (1.8,0) node(xline)[right]
        {$ $};
  \draw[axis] (0,-1.3) -- (0,1.5) node(yline)[above] {$ $};

  \draw[dashed] (0,-1.3) -- (0,1.5);

  \node[below ] at (-1.1,0)  {\tiny $-\frac{1}{2}$};
  \node[below right] at (0.9,0)  {\tiny $\frac{1}{2}$};	

  \node[above left] at (0.05,1)  {\tiny $\frac{1}{2}$};	
  \node[below left] at (0,-0.8)  {\tiny $-\frac{1}{2}$};	

  \node[above] at (-0.7,1)  {\tiny $F_2$};	
  \node[above right] at (0.45,0.45)  {\tiny $F_1$};	
  \node[left] at (-1,0.5)  {\tiny $F_3$};	
  
\end{tikzpicture}
\end{minipage}

    \begin{minipage}{.45\textwidth}
  \begin{tikzpicture}[scale=0.6]
 
    \clip (-3.4,-0.8) rectangle (3.8cm,3.6cm); 
  
  \coordinate (A1) at (0,0);
  \coordinate (A2) at (2,2);
  \coordinate (A3) at (-1,3);
   
 \draw[thick] (A1) -- (A2);
 \draw[thick, dashed]  (A2) -- (A3);	
 \draw[thick, lightgray]  (A3) -- (A1);

  \node[draw,circle,inner sep=1.1pt,fill] at (A1) {};
  \node[draw,circle,inner sep=1.1pt,fill] at (A2) {};
  \node[draw,circle,inner sep=1.1pt,fill] at (A3) {};
  
  \node[below] at (A1)  {\small $a(0,0)$};	
  \node[right] at (A2)  {\small $b(2,2)$};	
  \node[left] at (A3)  {\small $c(-1,3)$};
  
  \end{tikzpicture}
\end{minipage}
\end{tabular}

\caption{The unit ball of a polyhedral norm and a bar-joint framework with three induced monochrome subgraphs.}
\label{Additive}
 \end{figure}
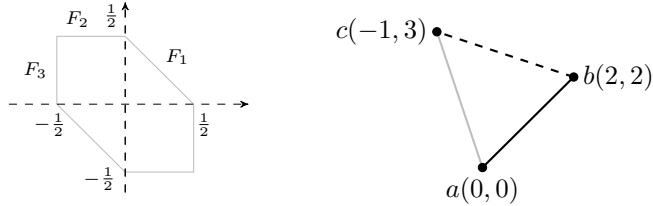

\subsection{Edge-labelled paths and relative infinitesimal rigidity}

For each edge $vw\in E(G)$ let $X_{vw}$ be the vector subspace of $\bR^d$,
\[X_{vw} = \bigcap_{\varphi_F\in\supp_{\Phi}(vw)} \ker \varphi_F
= \bigcap_{[F]\in\Phi(vw)} \ker \varphi_F\]
If $\gamma=\{v_1v_2,v_2v_3,\ldots,v_{n-1}v_n\}$
is a path in $G$ from a vertex $v_1$ to a vertex $v_n$ then we define,
 \[X_{\gamma} =  X_{v_1v_2}+X_{v_2v_3}+\cdots+X_{v_{n-1}v_n}\]
For each pair of vertices $v,w\in V(G)$ denote by $\Gamma_G(v,w)$  the set of all paths 
$\gamma$ in $G$ from $v$ to $w$.



A subframework of $(G,p)$ is a bar-joint framework $(H,p)$ obtained by restricting $p$ to the vertex set of a subgraph $H$.
We say that $(H,p)$ is {\em relatively infinitesimally rigid} in $(G,p)$ if the restriction of every infinitesimal flex of $(G,p)$ to $(H,p)$ is trivial.


\begin{proposition}
\label{RelRigid}
Let $(G,p)$ be a  bar-joint framework in $(\mathbb{R}^d,\|\cdot\|_\P)$ 
and let $(H,p)$ be a subframework of $(G,p)$.
If  for each pair of vertices $v,w\in V(H)$, \[\bigcap_{\gamma\in \Gamma_G(v,w)} X_{\gamma} = \{0\}\]
 then 
$(H,p)$ is relatively infinitesimally rigid in $(G,p)$.
\end{proposition}

\begin{proof}
Let  $u\in\F(G,p)$ be an infinitesimal flex of $(G,p)$ and let $v,w\in V(H)$.
Suppose $\gamma\in\Gamma_G(v,w)$ where 
$\gamma=\{v_1v_2,\ldots,v_{n-1}v_n\}$ is a path in $G$ with $v=v_1$ and $w=v_n$.
Then by Theorem \ref{flex1},
\[u_v-u_w=(u_{v_1}-u_{v_2})+(u_{v_2}-u_{v_3})+\cdots+(u_{v_{n-1}}-u_{v_n})\in X_{\gamma}\]
Since this holds for all paths in $\Gamma_G(v,w)$ the hypothesis implies that $u_v=u_w$.
Applying this argument to every pair of vertices in $H$ we see that the restriction of $u$ to $V(H)$ is constant and hence a trivial infinitesimal flex of $(H,p)$.
Thus $(H,p)$ is relatively infinitesimally rigid in $(G,p)$.

\end{proof}

\begin{example}
Let $(G,p)$ be the bar-joint framework in $(\bR^2,\|\cdot\|_\infty)$ indicated in Figure \ref{RelRigidPic}
and let $H$ be the subgraph of $G$ induced by the vertices $v_1,v_2,v_3$. The monochrome subgraphs $G_{F_1}$ and $G_{F_2}$ are indicated in black and gray respectively.
Each pair of vertices in $H$ is connected by a path in $G_{F_1}$ and a path in $G_{F_2}$
and so, by Proposition \ref{RelRigid}, $(H,p)$ is  relatively infinitesimally rigid in $(G,p)$.
\end{example}

 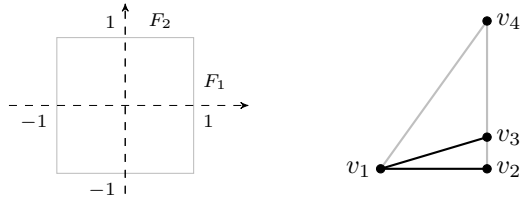
\begin{figure}[ht]
\centering
  \begin{tabular}{  c   }
  
 \begin{minipage}{.45\textwidth}
\begin{tikzpicture}[scale=0.9, axis/.style={very thin,dashed, ->, >=stealth'},
    important line/.style={very thick},
    dashed line/.style={dashed,  thick},
    pile/.style={thick, ->, >=stealth', shorten <=2pt, shorten
    >=2pt},
    every node/.style={color=black}]
 
  \clip (-3.5,-1.5) rectangle (2cm,1.5cm); 
  
  \coordinate (A1) at (-1,1);
  \coordinate (A2) at (0,0);
  \coordinate (A3) at (1,-1);
  \coordinate (A4) at (1,1);
  \coordinate (A5) at (-1,-1); 

  \draw[lightgray] (A3) -- (A4) -- (A1) -- (A5) -- cycle;
    
  \draw[axis] (-1.7,0)  -- (1.8,0) node(xline)[right]
        {$ $};
  \draw[axis] (0,-1.3) -- (0,1.5) node(yline)[above] {$ $};

  \draw[dashed] (0,-1.3) -- (0,1.5);

  \node[below left] at (-1,0)  {\tiny $-1$};
  \node[below right] at (1,0)  {\tiny $1$};	

  \node[above left] at (0,1)  {\tiny $1$};	
  \node[below left] at (0,-1)  {\tiny $-1$};	

  \node[above right] at (0.2,1)  {\tiny $F_2$};	
  \node[above right] at (1,0.1)  {\tiny $F_1$};	
\end{tikzpicture}
\end{minipage}

    \begin{minipage}{.45\textwidth}

  \begin{tikzpicture}[scale=0.7]
 
    \clip (-1,-0.6) rectangle (2.7cm,3cm); 
  
  \coordinate (A1) at (0,0);
  \coordinate (A2) at (2,0);
  \coordinate (A3) at (2,0.6);
  \coordinate (A4) at (2,2.8);

 \draw[thick] (A1) -- (A2) ;
 \draw[thick] (A1) -- (A3);	
 \draw[thick, lightgray]  (A1) -- (A4) -- (A3) -- (A2);

  \node[draw,circle,inner sep=1.1pt,fill] at (A1) {};
  \node[draw,circle,inner sep=1.1pt,fill] at (A2) {};
  \node[draw,circle,inner sep=1.1pt,fill] at (A3) {};
 \node[draw,circle,inner sep=1.1pt,fill] at (A4) {};

  \node[left] at (A1)  {\small $v_1$};
  \node[right] at (A2)  {\small $v_2$};	
  \node[right] at (A3)  {\small $v_3$};
  \node[right] at (A4)  {\small $v_4$};	
  
  \end{tikzpicture}
\end{minipage}
\end{tabular}

\caption{A relatively infinitesimally rigid subframework in $(\bR^2,\|\cdot\|_\infty)$.}
\label{RelRigidPic}
 \end{figure}

\begin{corollary}
\label{Ray}
Let $(G,p)$ be a bar-joint framework in $(\mathbb{R}^d,\|\cdot\|_\P)$ and let
$u\in \F(G,p)$ be an infinitesimal flex. 
If $vw\in E(G)$ and $p_v-p_w$ lies in a ray passing through an extreme point of $\P$ then $u_v=u_w$.
\end{corollary}

\begin{proof}
Let $x_0\in ext(\P)$ and suppose $p_v-p_w\in\{\lambda x_0:\lambda>0\}$.
Then \[\bigcap_{\gamma\in\Gamma_G(v,w)} X_{\gamma}\subseteq X_{vw}=
\{0\}\]
The result now follows from Proposition \ref{RelRigid}.
\end{proof}

\begin{example}
\label{RayEx}
Let $(K_{1,n},p)$ be a placement of the bipartite graph $K_{1,n}$ with edges $v_0v_1,v_0v_2,\ldots,v_0v_n$ in $(\bR^d,\|\cdot\|_\P)$ such that  $v_0$ is placed at the origin and all other vertices are placed at extreme points of $\P$. This bar-joint  framework is not well-positioned as each edge has at least $d$ distinct framework colours.
It follows from Corollary \ref{Ray} that $(K_{1,n},p)$ is infinitesimally rigid (but continuously flexible). 
Consider for example the class of polyhedral norms  on $\bR^2$ for which $\P$ is 
an $n$-gon with extreme points
$v_k = \left(\cos\left(\frac{2\pi (k-1)}{n}\right),\,\sin\left(\frac{2\pi (k-1)}{n}\right)\right)$
where $n\in2\bZ$, $n\geq 4$ and $k=1,2,\ldots,n$. 
Then $\P$ has $n$ facets $F_1,F_2,\ldots,F_n$ where $F_k$ is the closed line segment from $v_{k}$ to $v_{k+1}$.
Applying (\ref{RigMatrixEntries1}), the corresponding extreme point of the polar 
$\P^\triangle$ is 
$\hat{F_k} = \sec\left(\frac{\pi}{n}\right) \left(\cos\left(\frac{(2k-1)\pi}{n}\right), \, \sin\left(\frac{(2k-1)\pi}{n}\right)\right)$.
The case $n=8$ is illustrated in Figure \ref{RayFig} with $\P$ an octagon in $\bR^2$.
Each edge contributes two independent rows to the rigidity matrix $R_\P(K_{1,8},p)$. For example the entries for the row $v_0v_1$ are,
{\small $
 \kbordermatrix{
& v_{0,1} & v_{0,2} && v_{1,1} & v_{1,2} &  &  v_{2,1} & v_{2,2}
&    \cdots &      v_{8,1} & v_{8,2}\\
(v_0v_1,F_1)& 1  & \sqrt{2}-1 & \vrule & -1 & 1-\sqrt{2} & \vrule & 0 & 0 & \,\, \cdots \,\, &  0 & 0\\
(v_0v_1,F_8)& \sqrt{2}-1 & 1 & \vrule & 1-\sqrt{2} & -1 & \vrule &  0 & 0&  \cdots  &  0 & 0}
$}
\end{example}
In particular, the rigidity matrix for $(K_{1,8},p)$ has rank $2|E(K_{1,8})|=2|V(K_{1,8})|-2$ (cf. Proposition \ref{RigidityMatrixProp}).
\newdimen\R
\R=0.8cm

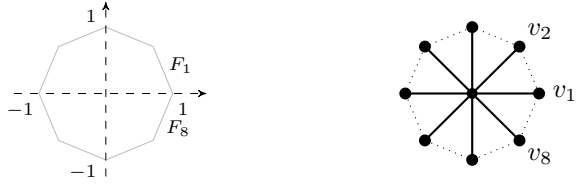
\begin{figure}[ht]
\centering
   \begin{tabular}{  c   }
 \begin{minipage}{.45\textwidth}
\begin{tikzpicture}[scale=1.1, axis/.style={very thin,dashed, ->, >=stealth'},
    important line/.style={very thick},
    dashed line/.style={dashed,  thick},
    pile/.style={thick, ->, >=stealth', shorten <=2pt, shorten
    >=2pt},
    every node/.style={color=black}]
 
  \clip (-2.5,-1.3) rectangle (2cm,1.2cm); 
  
            
        %
        \draw[lightgray] (0:\R) \foreach \x in {45,90,...,360} {
                -- (\x:\R) node {}
            } -- cycle (90:\R); 


  \draw[axis] (-1.1,0)  -- (1.2,0) node(xline)[right]
        {$ $};
  \draw[axis] (0,-1) -- (0,1.1) node(yline)[above] {$ $};

  \node[right] at (0.64,0.34) {\tiny $F_1$};
  \node[right] at (0.61,-0.42) {\tiny $F_8$};
  \node[below left] at (-0.75,0)  {\tiny $-1$};
  \node[below right] at (0.75,0)  {\tiny $1$};	

  \node[above left] at (0,0.75)  {\tiny $1$};	
  \node[below left] at (0,-0.75)  {\tiny $-1$};	

\end{tikzpicture}
\end{minipage}

    \begin{minipage}{.45\textwidth}

\begin{tikzpicture}[scale=1.1, axis/.style={very thin,dashed, ->, >=stealth'},
    important line/.style={very thick},
    dashed line/.style={dashed,  thick},
    pile/.style={thick, ->, >=stealth', shorten <=2pt, shorten
    >=2pt},
    every node/.style={color=black}]
 
  \clip (-1.6,-1.3) rectangle (2cm,1.2cm); 
  
            
        %
        \draw[dotted] (0:\R) \foreach \x in {45,90,...,360} {
                -- (\x:\R) node {}
            } -- cycle (90:\R); 
 \draw[thick,black]  \foreach \x in {0,45,90,...,359} {
               (0,0) -- (\x:\R) node[draw,circle,inner sep=1.3pt,fill] {}
            };
 \node[draw,circle,inner sep=1.3pt,fill] at (0,0) {};



  \node[right] at (0.85,0) {$v_1$};
  \node[above right] at (0.55,0.55)  {$v_2$};	
  \node[below right] at (0.55,-0.55)  {$v_8$};	

\end{tikzpicture}
\end{minipage}
\end{tabular}
 \caption{An infinitesimally rigid placement of the bipartite graph $K_{1,8}$ in $(\bR^2,\|\cdot\|_{\P})$.}
\label{RayFig}
\end{figure}

\subsection{Monochrome spanning subgraphs}
Applying the results of the previous sections we can now characterise the infinitesimally rigid bar-joint frameworks in $(\bR^d,\|\cdot\|_\P)$ which use exactly $d$ framework colours.

\begin{theorem}
\label{RigThm}
 Let $(G,p)$ be a  bar-joint framework in $(\mathbb{R}^d,\|\cdot\|_\P)$ and 
suppose that $|\Phi(G,p)|=d$. 
Then the following statements are equivalent.
\begin{enumerate}[(i)]
\item $(G,p)$ is  infinitesimally rigid.
\item $G_F$ contains a spanning tree of $G$ for each $[F]\in \Phi(G,p)$.
\end{enumerate}
\end{theorem}

\begin{proof}
The implication $(i)\Rightarrow (ii)$ follows from Proposition \ref{Polytope2}.
To prove $(ii)\Rightarrow (i)$ let $u\in\F(G,p)$.
If $v,w\in V(G)$ then for each framework colour $[F]\in\Phi(G,p)$ there exists a path in $G_F$ from $v$ to $w$.
Hence \[\bigcap_{\gamma\in\Gamma_G(v,w)} X_{\gamma} 
\subseteq \bigcap_{[F]\in \Phi(G,p)} \ker \varphi_{F}=\{0\}\]
and, by Proposition \ref{RelRigid}, $u_v=u_w$.
Applying this argument to all pairs $v,w\in V(G)$ we see that $u$ is a trivial infinitesimal flex  and so  $(G,p)$ is  infinitesimally rigid.
\end{proof}

A bar-joint framework $(G,p)$ is {\em minimally infinitesimally rigid} in $(\mathbb{R}^d,\|\cdot\|_\P)$  if it is infinitesimally rigid and every subframework obtained by removing a single edge from $G$ is infinitesimally flexible. 

\begin{corollary}
\label{MinRigCor}
 Let $(G,p)$ be a  bar-joint framework in $(\mathbb{R}^d,\|\cdot\|_\P)$ and suppose that 
$|\Phi(G,p)|=d$.
If $G_F$ is a spanning tree in $G$ for each $[F]\in \Phi(G,p)$
then $(G,p)$ is minimally infinitesimally rigid.
\end{corollary}

\begin{proof}
By Theorem \ref{RigThm}, $(G,p)$ is infinitesimally rigid. If any edge $vw$ is removed from $G$ then $G_F$ is no longer a spanning tree for some $[F]\in\Phi(G,p)$. By  Theorem \ref{RigThm},
the subframework $(G\backslash \{vw\},p)$ is not infinitesimally rigid and so we conclude that  $(G,p)$ is minimally infinitesimally rigid.
\end{proof}

\begin{example}
Let $(G,p)$ be the bar-joint framework in $(\bR^2,\|\cdot\|_\infty)$ which is illustrated 
in Figure \ref{MinRigidFig}. 
This bar-joint framework is well-positioned with respect to $\P$ and 
the subgraphs $G_{F_1}$ and  $G_{F_2}$ are indicated in black and gray respectively. 
Both monochrome subgraphs are spanning trees of $G$ and so, by Corollary \ref{MinRigCor}, $(G,p)$ is minimally infinitesimally rigid.
\end{example}

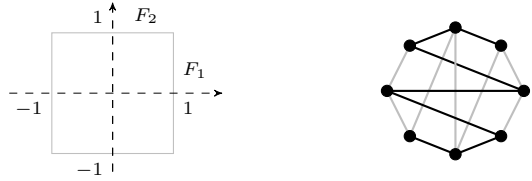
\begin{figure}[ht]
\centering
  \begin{tabular}{  c   }
  
 \begin{minipage}{.45\textwidth}
\begin{tikzpicture}[scale=0.8, axis/.style={very thin,dashed, ->, >=stealth'},
    important line/.style={very thick},
    dashed line/.style={dashed,  thick},
    pile/.style={thick, ->, >=stealth', shorten <=2pt, shorten
    >=2pt},
    every node/.style={color=black}]
 
  \clip (-3.5,-1.5) rectangle (2cm,1.5cm); 
  
  \coordinate (A1) at (-1,1);
  \coordinate (A2) at (0,0);
  \coordinate (A3) at (1,-1);
  \coordinate (A4) at (1,1);
  \coordinate (A5) at (-1,-1); 

  \draw[lightgray] (A3) -- (A4) -- (A1) -- (A5) -- cycle;
    
  \draw[axis] (-1.7,0)  -- (1.8,0) node(xline)[right]
        {$ $};
  \draw[axis] (0,-1.3) -- (0,1.5) node(yline)[above] {$ $};

  \draw[dashed] (0,-1.3) -- (0,1.5);

  \node[below left] at (-1,0)  {\tiny $-1$};
  \node[below right] at (1,0)  {\tiny $1$};	

  \node[above left] at (0,1)  {\tiny $1$};	
  \node[below left] at (0,-1)  {\tiny $-1$};	

  \node[above right] at (0.2,1)  {\tiny $F_2$};	
  \node[above right] at (1,0.1)  {\tiny $F_1$};	
\end{tikzpicture}
\end{minipage}

    \begin{minipage}{.45\textwidth}
  \begin{tikzpicture}[ scale=0.6]
 
    \clip (-1.5,-0.8) rectangle (8cm,2.7cm); 
  
  \coordinate (Aone) at (0,0);
  \coordinate (Atwo) at (1,-0.4);
  \coordinate (Athree) at (2,0);
  \coordinate (Afour) at (2.5,1);
  \coordinate (Afive) at (2,2);
  \coordinate (Asix) at (1,2.4);
  \coordinate (Aseven) at (0,2);
  \coordinate (Aeight) at (-0.5,1);

 \draw[thick] (Aone) -- (Atwo) -- (Athree);
 \draw[thick,lightgray] (Athree) -- (Afour) -- (Afive);
 \draw[thick] (Afive) -- (Asix) -- (Aseven);
 \draw[thick,lightgray] (Aseven) -- (Aeight) -- (Aone);

 \draw[thick,lightgray] (Aone) -- (Asix); 

 \draw[thick,lightgray] (Atwo) -- (Afive); 
 \draw[thick] (Afour) -- (Aseven); 

 \draw[thick,lightgray] (Atwo) -- (Asix); 
 \draw[thick] (Athree) -- (Aeight); 


 \draw[thick] (Afour) -- (Aeight);

      \node[draw,circle,inner sep=1.5pt,fill] at (Aone) {};
  \node[draw,circle,inner sep=1.5pt,fill] at (Atwo) {};
  \node[draw,circle,inner sep=1.5pt,fill] at (Athree) {};
  \node[draw,circle,inner sep=1.5pt,fill] at (Afour) {};
  \node[draw,circle,inner sep=1.5pt,fill] at (Afive) {};
  \node[draw,circle,inner sep=1.5pt,fill] at (Asix) {};
  \node[draw,circle,inner sep=1.5pt,fill] at (Aseven) {};
  \node[draw,circle,inner sep=1.5pt,fill] at (Aeight) {};

  \end{tikzpicture}
\end{minipage}
\end{tabular}

\caption{A minimally infinitesimally rigid bar-joint framework in $(\bR^2,\|\cdot\|_\infty)$}
\label{MinRigidFig}
\end{figure}

The converse statement to Corollary \ref{MinRigCor} which we now prove requires the additional assumption that $(G,p)$ is well-positioned. The necessity of this condition is demonstrated in Example \ref{3dEx}.  

\begin{corollary}
\label{MinRigCor2}
 Let $(G,p)$ be a well-positioned bar-joint framework in $(\mathbb{R}^d,\|\cdot\|_\P)$ and suppose that 
$|\Phi(G,p)|=d$.
Then the following statements are equivalent.
\begin{enumerate}[(i)]
\item $(G,p)$ is minimally infinitesimally rigid.
\item $G_F$ is a spanning tree in $G$ for each $[F]\in \Phi(G,p)$. 
\end{enumerate}
\end{corollary}

\begin{proof}
$(i)\Rightarrow (ii)$. Let $[F]\in\Phi(G,p)$. 
If $(G,p)$ is minimally infinitesimally rigid then by Theorem \ref{RigThm},
the monochrome subgraph  $G_F$ contains a spanning tree of $G$.
Suppose $vw$ is an edge of $G$ which is contained in $G_F$.
Since $(G,p)$ is minimally infinitesimally rigid, $(G\backslash \{vw\},p)$ is infinitesimally flexible.
Since $(G,p)$ is well-positioned,  $vw$ is contained in exactly one monochrome subgraph of $G$ and so $G_F$ is the only monochrome subgraph which is altered by removing the edge $vw$ from $G$.
By Theorem \ref{RigThm}, $G_F\backslash\{vw\}$ does not contain a spanning tree of $G$.
We conclude that $G_F$ is a spanning tree of $G$. 
The implication $(ii)\Rightarrow(i)$ is proved in Corollary \ref{MinRigCor}. 

\end{proof}




\begin{example}
\label{3dEx}
Let $(G,p)$ be the bar-joint framework in $(\bR^3,\|\cdot\|_\infty)$ which is illustrated in Figure \ref{3dExPic}. The polytope $\P$ is the cube with extreme points
$\pm(1,1,1)$, $\pm(1,1,-1)$, $\pm(1,-1,1)$, $\pm(-1,1,1)$
and the polyhedral norm is the maximum norm,
\[\|x\|_\P =\max_{i=1,2,3} |x_i|= \|x\|_\infty\]
This bar-joint framework is not well-positioned as each edge has two framework colours, 
\[\Phi(ab) = \{[F_1],[F_2]\},\,\,\,\, \,\,\,\,\,\Phi(ac) = \{[F_1],[F_2]\},\,\, \,\,\,\,\,\,\,
\Phi(ad) = \{[F_2],[F_3]\}\]
\[\Phi(bd) = \{[F_1],[F_3]\},\,\,\,\, \,\,\,\,\,
\Phi(cd) = \{[F_1],[F_3]\}\]
The monochrome subgraphs $G_{F_1}$, $G_{F_2}$ and $G_{F_3}$ are indicated in blue, red and green respectively and each contains a spanning tree of $G$.
Thus by Theorem \ref{RigThm}, $(G,p)$ is infinitesimally rigid. 
The corresponding extreme points of $\P^\triangle$ are $\hat{F}_1 = (1,0,0)$,
$\hat{F}_2 = (0,1,0)$ and $\hat{F}_3 = (0,0,1)$.
The rigidity matrix $R_\P(G,p)$ is 
\[ \kbordermatrix{
& a_x & a_y & a_z && b_x & b_y & b_z & &  c_x & c_y & c_z & d_x & d_y & d_z \\
(ab,F_1) &1  & 0 &0  &\vrule & -1 & 0 & 0 & \vrule & 0 & 0 & 0 &\vrule & 0 & 0 & 0 \\
(ab,F_2) &0 & 1 & 0 &\vrule & 0 & -1 & 0 & \vrule &  0 & 0 & 0&\vrule &0 & 0 & 0\\
(ac,F_1)& 1  & 0 &0& \vrule& 0 & 0 & 0& \vrule &  -1 & 0 &0& \vrule & 0 & 0 & 0\\
(ac,F_2)&0 & 1 & 0 &\vrule & 0 & 0 & 0& \vrule  & 0& -1 &0 &\vrule & 0 & 0 & 0\\
(ad,F_2)&0 & 1 & 0&\vrule  &0 & 0 & 0& \vrule & 0 & 0 & 0&\vrule &0& -1 &0\\
(ad,F_3)& 0 & 0 & 1&\vrule & 0 & 0 & 0& \vrule &  0 & 0 & 0&\vrule &0 &0& -1\\
(bd,F_1)& 0 & 0 & 0& \vrule &1  & 0 &0 & \vrule &  0 & 0 & 0&\vrule &  -1 & 0 & 0\\
(bd,F_3)&0 & 0 & 0 &\vrule & 0 & 0 & 1& \vrule&  0 & 0 & 0&\vrule & 0 &0& -1\\
(cd,F_1)& 0 & 0 & 0& \vrule & 0 & 0 & 0& \vrule & 1  & 0 &0 &\vrule&  -1 & 0 & 0\\
(cd,F_3)& 0 & 0 & 0& \vrule & 0 & 0 & 0& \vrule &0 & 0 & 1&\vrule & 0 &0& -1
}\]
The rigidity matrix has rank $3|V(G)|-3=9$. By removing the edge $ad$ the resulting rigidity matrix has rank $7$ and so the subframework $(G\backslash \{ad\},p)$ has infinitesimal flex dimension $2$. Removing any other edge results in a subframework with infinitesimal flex dimension $1$.
Hence $(G,p)$ is minimally infinitesimally rigid.
However, $G_{F_1}$ is not itself a spanning tree and this demonstrates the necessity in the hypothesis of Corollary \ref{MinRigCor2} that $(G,p)$ is well-positioned.
\end{example}

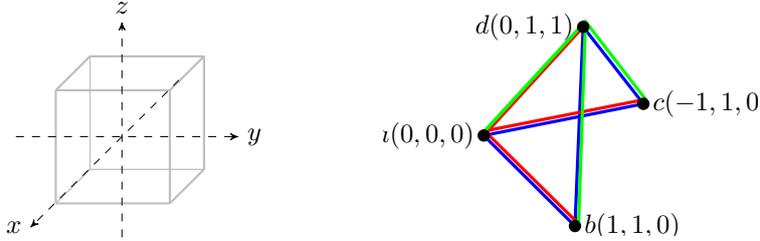
\begin{figure}[ht]
\centering
  \begin{tabular}{  c   }
  
 \begin{minipage}{.45\textwidth}
\begin{tikzpicture}[scale=1.5, axis/.style={very thin,dashed, ->, >=stealth'},
    important line/.style={very thick},
    dashed line/.style={dashed,  thick},
    pile/.style={thick, ->, >=stealth', shorten <=2pt, shorten
    >=2pt},
    every node/.style={color=black}]
 
  \clip (-1,-0.5) rectangle (1.9cm,1.9cm); 
    \coordinate (A1) at (0, 0);
    \coordinate (A2) at (0, 1);
    \coordinate (A3) at (1, 1);
    \coordinate (A4) at (1, 0);
    \coordinate (B1) at (0.3, 0.3);
    \coordinate (B2) at (0.3, 1.3);
    \coordinate (B3) at (1.3, 1.3);
    \coordinate (B4) at (1.3, 0.3);

    \draw[ thick,lightgray] (A1) -- (A2);
    \draw[thick,lightgray] (A2) -- (A3);
    \draw[ thick,lightgray] (A3) -- (A4);
    \draw[thick,lightgray] (A4) -- (A1);

    \draw[lightgray] (A1) -- (B1);
    \draw[lightgray] (B1) -- (B2);
    \draw[ thick,lightgray] (A2) -- (B2);
    \draw[ thick,lightgray] (B2) -- (B3);
    \draw[ thick,lightgray] (A3) -- (B3);
    \draw[ thick,lightgray] (A4) -- (B4);
    \draw[ thick,lightgray] (B4) -- (B3);
    \draw[lightgray] (B1) -- (B4);

    \draw[lightgray] (A1) -- (B1) -- (B4) -- (A4);
    \draw[lightgray] (A1) -- (A2) -- (A3) -- (A4);
    \draw[lightgray] (A1) -- (A2) -- (B2) -- (B1);
    \draw[lightgray] (B1) -- (B2) -- (B3) -- (B4);
    \draw[lightgray] (A3) -- (B3) -- (B4) -- (A4);
    \draw[lightgray] (A2) -- (B2) -- (B3) -- (A3);

  \draw[axis] (-0.35,0.59)  -- (1.6,0.59) node(xline)[right]
        {$y $};
  \draw[axis] (0.58,-0.3) -- (0.58,1.6) node(yline)[above] {$z $};
 \draw[axis] (1.08,1.09) -- (-0.22,-0.21) node(zline)[left] {$x $};

\end{tikzpicture}
\end{minipage}

   \begin{minipage}{.45\textwidth}
  \begin{tikzpicture}[scale=0.6]
 
    \clip (-2.2,-2.2) rectangle (6cm,2.7cm); 
  
  \coordinate (A1) at (0,0);
  \coordinate (A2) at (2,-2);
  \coordinate (A3) at (3.5,0.7);
  \coordinate (A4) at (2.18,2.4);
 
 \draw[very thick, blue] (A1) -- (A2) -- (A4) -- (A3) -- cycle;
 \draw[very thick,red] (A1) -- (A4);

  \coordinate (B1) at (0.02,0.1);
  \coordinate (B2) at (2.08,-1.95);
  \coordinate (B3) at (3.55,0.8);
  \coordinate (B4) at (2.23,2.52);
 
 \draw[very thick,green] (B1) -- (B4) -- (B3);
 \draw[very thick,red] (B1) -- (B2);
 \draw[very thick,red] (B1) -- (B3);
 \draw[very thick,green] (B1) -- (B4);
 \draw[very thick,green] (B4) -- (B2);

      \node[draw,circle,inner sep=1.5pt,fill] at (A1) {};
  \node[draw,circle,inner sep=1.5pt,fill] at (A2) {};
  \node[draw,circle,inner sep=1.5pt,fill] at (A3) {};
  \node[draw,circle,inner sep=1.5pt,fill] at (A4) {};

  \node[left] at (A1)  {\small $a(0,0,0)$};
  \node[right] at (A2)  {\small $b(1,1,0)$};	
  \node[right] at (A3)  {\small $c(-1,1,0)$};	
  \node[ left] at (A4)  {\small $d(0,1,1)$};	

  \end{tikzpicture}
\end{minipage}
\end{tabular}

 \caption{A minimally  infinitesimally rigid bar-joint framework in $(\bR^3,\|\cdot\|_{\infty})$}
\label{3dExPic}
\end{figure}

\section{An analogue of Laman's theorem}
\label{Sparsity}
In this section we address the problem of whether there exists a combinatorial description of the class of graphs for which a minimally infinitesimally rigid placement exists in $(\bR^d,\|\cdot\|_\P)$.
We restrict our attention to finite bar-joint frameworks and prove that in two dimensions such a characterisation exists (Theorem \ref{Laman}). This result is analogous to Laman's theorem \cite{Lam} for bar-joint frameworks in the Euclidean plane and extends \cite[Theorem 4.6]{kit-pow} which holds in the case where $\P$ is a quadrilateral. 
\subsection{Regular placements}
Let $\omega(G,\bR^d,\P)$ denote the set of all well-positioned  placements of a finite simple graph $G$ in $(\bR^d,\|\cdot\|_\P)$. 
A bar-joint framework $(G,p)$ is  {\em regular  in $(\mathbb{R}^d,\|\cdot\|_\P)$} if the function 
\[\omega(G,\bR^d,\P)\to\{1,2,\ldots,d|V(G)|-d\}, \,\,\,\,\,\, x\mapsto \rank R_\P(G,x)\] 
achieves its maximum value at $p$.

\begin{lemma} Let $G$ be a finite simple graph.
\begin{enumerate}[(i)]
\item The set of placements of  $G$ in $(\bR^d,\|\cdot\|_\P)$ which are both well-positioned and regular  is  an open set in $\prod_{v\in V(G)}\mathbb{R}^{d}$.
\item The set of placements of  $G$ in $(\bR^d,\|\cdot\|_\P)$ which are  well-positioned  and not regular  is  an open set in $\prod_{v\in V(G)}\mathbb{R}^{d}$.
\end{enumerate}
\end{lemma}

\begin{proof}
Let $p$ be a well-positioned placement of $G$ and let $U$ be an open neighbourhood of $p$ as in the statement of Proposition \ref{Configuration}.
The matrix-valued function $x\mapsto R_\P(G,x)$ is constant on $U$ and so either $(G,x)$ is regular for all $x\in U$ or $(G,x)$ is not regular for all $x\in U$. 
\end{proof}

A finite simple graph $G$ is {\em (minimally) rigid}  in $(\bR^d,\|\cdot\|_\P)$ if 
there exists a well-positioned placement of $G$ which is (minimally) infinitesimally rigid.

\begin{example}
\label{K4Rigid}
The complete graph $K_4$ is minimally rigid in $(\bR^2,\|\cdot\|_\P)$ for every polyhedral norm $\|\cdot\|_\P$.
To see this let $F_1,F_2,\ldots,F_n$ be the facets of $\P$ and let
$x_0\in ext(\P)$ be any extreme point of $\P$. 
Then $x_0$ is contained in exactly two facets, $F_1$ and $F_2$ say.
Choose a point $x_1$ in the relative interior of $F_1$ and a point $x_2$ in the relative interior of $F_2$. 
Then by formulas (\ref{norm2}) and (\ref{norm3}),
\begin{eqnarray}
\label{K4eqn1}
\max_{k\not=1}\, (x_1\cdot \hat{F}_k) <\|x_1\|_\P= x_1\cdot \hat{F}_1=1
\end{eqnarray}
\begin{eqnarray}
\label{K4eqn2}
\max_{k\not=2}\, (x_2\cdot \hat{F}_k) < \|x_2\|_\P=x_2\cdot \hat{F}_2=1
\end{eqnarray}
Since $(x_0\cdot \hat{F}_1) = (x_0\cdot \hat{F}_2)=\|x_0\|_\P=1$, if $x_1$ and $x_2$ are chosen to lie in a sufficiently small neighbourhood of  $x_0$ then by continuity we may assume,
\begin{eqnarray}
\label{K4eqn3}
x_1\cdot\hat{F}_2=\max_{k\not=1}\, (x_1\cdot\hat{F_k})>0
\end{eqnarray}
\begin{eqnarray}
\label{K4eqn4}
x_2\cdot\hat{F}_1=\max_{k\not=2}\, (x_2\cdot\hat{F_k})>0
\end{eqnarray}
We may also assume without loss of generality that 
\begin{eqnarray}
\label{K4eqn5}
x_1\cdot\hat{F}_2 = x_2\cdot \hat{F}_1
\end{eqnarray}
Define a placement $p:V(K_4)\to\bR^2$ by setting
\[p_{v_0} = (0,0), \,\,\,\,\, p_{v_1}=x_1, \,\,\,\,\, p_{v_2} = (1-\epsilon)x_2,
\,\,\,\,\, p_{v_3}=x_1+(1+\epsilon)x_2\]
where $0<\epsilon<1$.
The edges $v_0v_1$, $v_0v_2$ and $v_1v_3$ have framework colours,
\[\Phi(v_0v_1)=[F_1],  \,\,\,\,\,\, \Phi(v_0v_2) = [F_2],\,\,\,\,\,\,\Phi(v_1v_3)=[F_2]\]
To determine the framework colours for the remaining edges we will apply the above identities together with formulas (\ref{norm2}) and (\ref{norm3}).

Consider the edge $v_2v_3$. 
If $k\not=1$ and $\epsilon$ is sufficiently small then applying (\ref{K4eqn1}), 
\[(p_{v_3}-p_{v_2})\cdot \hat{F}_k
= (x_1\cdot\hat{F}_k)+2\epsilon\, (x_2\cdot \hat{F}_k)
<1\]
Also by (\ref{K4eqn1}) and (\ref{K4eqn4}) we have,
 \[(p_{v_3}-p_{v_2})\cdot \hat{F}_1 =(x_1\cdot\hat{F}_1)+2\epsilon\, (x_2\cdot \hat{F}_1)=1+2\epsilon\, (x_2\cdot \hat{F}_1)>1 \]
We conclude that $F_1$ is the unique facet of $\P$ for which $\|p_{v_3}-p_{v_2}\|_\P=
(p_{v_3}-p_{v_2})\cdot \hat{F}_1$
and so $p_{v_3}-p_{v_2}\in\cone(F_1)^\circ$.
Thus $\Phi(v_2v_3)=[F_1]$.

Consider the edge $v_0v_3$. 
Applying (\ref{K4eqn2}) and (\ref{K4eqn3}), for $k\not=1,2$ we have, 
\[(p_{v_3}-p_{v_0})\cdot \hat{F}_k
=(x_1\cdot\hat{F}_k)+(1+\epsilon)\, (x_2\cdot \hat{F}_k)
< (x_1\cdot\hat{F}_2)+1+\epsilon\]
By applying (\ref{K4eqn5}), 
\[(p_{v_3}-p_{v_0})\cdot \hat{F}_1 = (x_1\cdot\hat{F}_1)+(1+\epsilon)(x_2\cdot\hat{F}_1)<(x_1\cdot\hat{F}_2)+1+\epsilon \]
and by (\ref{K4eqn2}),
 \[(p_{v_3}-p_{v_0})\cdot \hat{F}_2 = (x_1\cdot\hat{F}_2)+(1+\epsilon)(x_2\cdot\hat{F}_2)=(x_1\cdot\hat{F}_2)+1+\epsilon\]
Hence $F_2$ is the unique facet of $\P$  for which $\|p_{v_3}-p_{v_0}\|_\P=
(p_{v_3}-p_{v_0})\cdot \hat{F}_2$. Thus $p_{v_3}-p_{v_0}\in\cone(F_2)^\circ$
and so
$\Phi(v_0v_3)=[F_2]$.

Finally, consider the edge $v_1v_2$. 
Applying (\ref{K4eqn5}) we have, 
\[(p_{v_2}-p_{v_1})\cdot \hat{F}_2
= (1-\epsilon)(x_2\cdot\hat{F}_2) - (x_1\cdot\hat{F}_2)
=1-\epsilon - (x_2\cdot \hat{F}_1)\]
and this value is positive provided $\epsilon$ is sufficiently small.
By (\ref{K4eqn1}) we have,
\[(p_{v_2}-p_{v_1})\cdot (-\hat{F}_1)=-(1-\epsilon)(x_2\cdot\hat{F}_1)+(x_1\cdot\hat{F}_1)
=1+\epsilon(x_2\cdot\hat{F}_1)-(x_2\cdot\hat{F}_1)\]
We conclude that $(p_{v_2}-p_{v_1})\cdot(\pm \hat{F}_2)<\|p_{v_2}-p_{v_1}\|_\P$.
Hence   $p_{v_2}-p_{v_1}\notin\cone(F_2)$ and so $\Phi(v_1v_2)=[F_k]$ for some $[F_k]\not=[F_2]$.

By making a small perturbation we can assume that $p_{v_2}-p_{v_1}$ is contained in the conical hull of exactly one facet of $\P$ and so $(G,p)$ is well-positioned.
This framework colouring is illustrated in Figure \ref{K4RigidFig} with monochrome subgraphs 
$G_{F_1}$ and $G_{F_2}$  indicated in  black and  gray respectively and $G_{F_k}$ indicated by the dotted line. 

Suppose $u\in\F(K_4,p)$. To show that $u$ is a trivial infinitesimal flex we apply the method of Proposition \ref{RelRigid}.
The vertices $v_0$ and $v_1$ are joined by monochrome paths in both $G_{F_1}$ and $G_{F_2}$ and so $u_{v_0}=u_{v_1}$. The vertices $v_2$ and $v_3$ are also joined by monochrome paths in both $G_{F_1}$ and $G_{F_2}$ and so $u_{v_2}=u_{v_3}$.
The vertices $v_1$ and $v_2$ are joined by monochrome paths in $G_{F_2}$ and $G_{F_k}$ and so $u_{v_1}=u_{v_2}$. Thus $u$ is a constant and hence trivial infinitesimal flex of $(K_4,p)$.
We conclude that $(K_4,p)$, and all regular and well-positioned placements of $K_4$, are 
infinitesimally rigid.
\end{example}

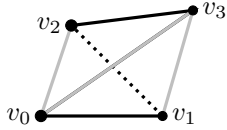
\begin{figure}[h]
\centering
  \begin{tabular}{  c   }
  
   \begin{tikzpicture}[scale=0.4, axis/.style={very thin,dashed, ->, >=stealth'},
    important line/.style={very thick},
    dashed line/.style={dashed,  thick},
    pile/.style={thick, ->, >=stealth', shorten <=2pt, shorten
    >=2pt},
    every node/.style={color=black}]
 
  \clip (-2,-0.5) rectangle (7.2cm,4cm); 
  
  \coordinate (A1) at (0,0);
  \coordinate (A2) at (4,0);
  \coordinate (A3) at (1,3);
  \coordinate (A4) at (5,3.5);

  \draw[very thick] (A3) -- (A4) -- (A1) -- (A2);
  \draw[ very thick,   dotted] (A2) -- (A3);
  \draw[very thick, lightgray] (A3) -- (A1) -- (A4) -- (A2);
    
	
  \node[draw,circle,inner sep=1.5pt,fill] at (A1) {};
  \node[draw,circle,inner sep=1.2pt,fill] at (A2) {};
  \node[draw,circle,inner sep=1.5pt,fill] at (A3) {};
  \node[draw,circle,inner sep=1.2pt,fill] at (A4) {};

  \node[ right] at (A4)  {\small $v_3$};
  \node[left] at (A3)  {\small $v_2$};	
  \node[ right] at (A2)  {\small $v_1$};	
  \node[ left] at (A1)  {\small $v_0$};	
 
\end{tikzpicture}   

  \end{tabular}
  \caption{A framework colouring for an infinitesimally rigid placement of $K_4$ in $(\bR^2,\|\cdot\|_\P)$}
\label{K4RigidFig}
\end{figure}

In Euclidean space it is often necessary to use a stronger notion of genericity for bar-joint frameworks which requires that all subframeworks of $(K_{V(G)},p)$ be regular frameworks.
Here $K_{V(G)}$ is the complete graph on the vertices of $G$. 
In the Euclidean setting (and more generally for the classical $\ell^p$ norms with $p\in(1,\infty)$),
such placements form an open and dense subset of $\prod_{v\in V(G)}\mathbb{R}^{d}$ (see for example \cite[Lemma 2.7]{kit-pow-2}).
The following example shows that in the case of polyhedral norms such placements need not exist.

\begin{example}
Consider a well-positioned placement $p$ of the complete graph $K_6$ in $(\bR^2,\|\cdot\|_\infty)$. 
The induced framework colouring of the edges of $K_6$ contains a monochrome subgraph $G_F$ which itself contains a copy of the complete graph $K_3$. The subframework $(K_3,p)$ has infinitesimal flex dimension $2$. Since the regular placements of $K_3$ have infinitesimal flex dimension $1$, $(K_3,p)$ is not regular. 
Thus there does not exist a well-positioned placement of $K_6$ in $(\bR^2,\|\cdot\|_\infty)$ for which all subframeworks are regular.
More generally, it follows from Ramsey's theorem that given any polyhedral norm on $\bR^d$ there exists a complete graph for which no such well-positioned placements exists. 
\end{example}

\subsection{Counting conditions}
The Maxwell counting conditions \cite{max} state that  a finite minimally infinitesimally rigid bar-joint  framework $(G,p)$ in Euclidean space $\bR^d$  must satisfy $|E(G)|=d|V(G)|-{d+1\choose 2}$ with inequalities $|E(H)|\leq d|V(H)|-{d+1\choose 2}$  for all subgraphs $H$.
The following  analogous statement holds for polyhedral norms.

\begin{proposition}
\label{Maxwell}
Let $(G,p)$ be a finite well-positioned bar-joint framework in $(\mathbb{R}^d,\|\cdot\|_\P)$.
If $(G,p)$ is minimally infinitesimally rigid then 
\begin{enumerate}[(i)]
\item $|E(G)|= d|V(G)|-d$, and,
\item $|E(H)|\leq d|V(H)|-d$
for all subgraphs $H$ of $G$.
\end{enumerate}
\end{proposition}

\begin{proof}
If $(G,p)$ is minimally infinitesimally rigid then by Proposition \ref{RigidityMatrixProp} the rigidity matrix $R_\P(G,p)$  is independent and,
\[|E(G)| = \rank R_\P(G,p) = d|V(G)|-d\]
The rigidity matrix for any subframework  of $(G,p)$ is also independent and so
\[|E(H)| = \rank R_\P(H,p) \leq d|V(H)|-d\]
for all subgraphs $H$.
\end{proof}

A graph $G$ is {\em $(d,d)$-tight} if it satisfies the counting conditions in the above proposition.
The class of $(2,2)$-tight graphs has the property that every member can be constructed from a single vertex by applying a sequence of finitely many allowable graph moves (see \cite{NOP,NOP2}). 
The allowable graph moves are:
\begin{enumerate}
\item The Henneberg 1-move (also called vertex addition, or $0$-extension).

\item The Henneberg 2-move (also called edge splitting, or $1$-extension).

\item The vertex splitting move.

\item The vertex-to-$K_4$ move.
\end{enumerate}

A Henneberg $1$-move $G\to G'$ adjoins a vertex $v_0$ to $G$ together with two edges $v_0v_1$ and $v_0v_2$ where $v_1,v_2\in V(G)$.

\begin{proposition}
\label{Graph1}
The Henneberg 1-move preserves infinitesimal rigidity for well-positioned bar-joint frameworks in $(\bR^2,\|\cdot\|_\P)$.
\end{proposition}

\begin{proof}
Suppose $(G,p)$ is well-positioned and infinitesimally rigid and let $G\to G'$ be a Henneberg 1-move on the vertices $v_1,v_2\in V(G)$.
Choose distinct $[F_1],[F_2]\in\Phi(\P)$  and define 
a placement $p'$ of $G'$ by  $p'_{v}=p_{v}$ for all $v\in V(G)$ and 
\[p'_{v_0}\in (p_{v_1}\pm\cone(F_1)^{\circ})\cap(p_{v_2}\pm\cone(F_2)^{\circ})\]
Then $(G',p')$ is well-positioned and the edges $v_0v_1$ and $v_0v_2$ have framework colours $[F_1]$ and $[F_2]$ respectively.
If $u\in\F(G',p')$ then the restriction of $u$ to $V(G)$ is an infinitesimal flex of $(G,p)$.
This restriction must be trivial and hence constant. In particular, $u_{v_1}=u_{v_2}$.
By Theorem \ref{flex1}, $\varphi_{F_1}(u_{v_0}-u_{v_1})=0$ and 
$\varphi_{F_2}(u_{v_0}-u_{v_1})=\varphi_{F_2}(u_{v_0}-u_{v_2})=0$
and so $u_{v_0}=u_{v_1}$. We conclude that $(G',p')$ is infinitesimally rigid.
\end{proof}

A Henneberg $2$-move $G\to G'$ removes an edge $v_1v_2$ from $G$ and adjoins a vertex $v_0$
together with three edges $v_0v_1$, $v_0v_2$ and $v_0v_3$. 

\begin{proposition}
\label{Graph2}
The Henneberg 2-move preserves infinitesimal rigidity for well-positioned bar-joint frameworks in $(\bR^2,\|\cdot\|_\P)$.
\end{proposition}

\begin{proof}
Suppose $(G,p)$ is well-positioned and infinitesimally rigid and let $G\to G'$ be a Henneberg 2-move on the vertices $v_1,v_2,v_3\in V(G)$ and the edge $v_1v_2\in E(G)$. Let $[F_1]$ be the unique  framework colour for the edge $v_1v_2$ and choose any    $[F_2]\in\Phi(\P)$ with $[F_2]\not=[F_1]$.
Define  a placement $p'$ of $G'$ by  setting $p'_{v}=p_{v}$ for all $v\in V(G)$ and choosing $p'_{v_0}$ to lie on the intersection of the line through $p_{v_1}$ and $p_{v_2}$ and  the double cone
$p_{v_3}\pm\cone(F_2)^\circ$.  
Then $(G',p')$ is well-positioned. The edges $v_0v_1$ and $v_0v_2$ both have framework colour $[F_1]$ and the edge $v_0v_3$ has framework colour $[F_2]$. 
If $u\in\F(G',p')$ then by Theorem \ref{flex1}, 
\[\varphi_{F_1}(u_{v_1}-u_{v_2})=\varphi_{F_1}(u_{v_1}-u_{v_0})+\varphi_{F_1}(u_{v_0}-u_{v_2})=0\]
Hence the restriction of $u$ to $V(G)$ is an infinitesimal flex of $(G,p)$ and must be trivial.
In particular, $u_{v_1}=u_{v_3}$. Now $\varphi_{F_1}(u_{v_0}-u_{v_1})=0$ and
$\varphi_{F_2}(u_{v_0}-u_{v_1})=\varphi_{F_2}(u_{v_0}-u_{v_3})=0$ and so $u_{v_0}=u_{v_1}$.
We conclude that $u$ is a constant and hence trivial infinitesimal flex of $(G',p')$.
\end{proof}

A vertex splitting move $G\to G'$ adjoins a new vertex $v_0$ and two new edges $v_0v_1$ and $v_0v_2$ to $G$ where $v_1v_2$ is an edge of $G$. Edges $v_1w$ of $G$ which are incident with $v_1$ may be replaced with the edge $v_0w$.

\begin{proposition}
\label{Graph3}
The vertex splitting move preserves infinitesimal rigidity for finite well-positioned bar-joint frameworks in $(\bR^2,\|\cdot\|_\P)$.
\end{proposition}

\begin{proof}
Suppose $(G,p)$ is well-positioned and infinitesimally rigid and let $G\to G'$ be a vertex splitting move on the vertex $v_1\in V(G)$ and the edge $v_1v_2\in E(G)$. Let $[F_1]$ be the unique framework colour for $v_1v_2$ and choose any  $[F_2]\in\Phi(\P)$ with $[F_2]\not=[F_1]$.
Since $v_1$ has finite valence, there exists an open ball $B(p_{v_1},r)$ such that if $p_{v_1}$ is replaced with any point $x\in B(p_{v_1},r)$ then the induced framework colouring of $G$ is left unchanged. 
Define  a placement $p'$ of $G'$ by setting $p'_{v}=p_{v}$ for all $v\in V(G)$
and choosing  \[p'_{v_0}\in (p_{v_1}+\cone(F_2)^{\circ})\cap B(p_{v_1},r)\]
Then $(G',p')$ is well-positioned. 
Suppose $u\in \F(G',p')$ is an infinitesimal flex of $(G',p')$.
The framework colours for the edges $v_0v_1$ and $v_0v_2$ are $[F_2]$ and $[F_1]$ respectively.
Thus there exists a path from $v_0$ to $v_1$ in the monochrome subgraph $G'_{F_1}$ given by the edges $v_1v_2,v_2v_0$ and there exists a path from $v_0$ to $v_1$ in the monochrome subgraph $G'_{F_2}$ 
given by the edge $v_0v_1$.
By the relative rigidity method of Proposition \ref{RelRigid}, $u_{v_0}=u_{v_1}$.
If an edge $v_1w$ in $G$ has framework colour $[F]$ induced by $(G,p)$ and is replaced by $v_0w$ in $G'$ then the framework colour is unchanged. 
Thus applying Theorem \ref{flex1}, 
\[\varphi_{F}(u_{v_1}-u_w) = \varphi_F(u_{v_1}-u_{v_0})+\varphi_F(u_{v_0}-u_w)=0\]
and so the restriction of $u$ to $V(G)$ is an infinitesimal flex of $(G,p)$. 
This restriction is constant since $(G,p)$ is infinitesimally rigid and so $u$ is a trivial infinitesimal flex of $(G',p')$. 
\end{proof}

A vertex-to-$K_4$ move $G\to G'$ replaces a vertex $v_0\in V(G)$ with a copy of the complete graph $K_4$ by adjoining three new vertices $v_1,v_2,v_3$ and six edges $v_0v_1$, $v_0v_2$, $v_0v_3$, $v_1v_2$, $v_1v_3$, $v_2v_3$.
Each edge $v_0w$ of $G$ which is incident with $v_0$  may be left unchanged or replaced by one of $v_1w$, $v_2w$ or $v_3w$.  

\begin{proposition}
\label{Graph4}
The vertex-to-$K_4$ move preserves infinitesimal rigidity for finite well-positioned bar-joint frameworks in $(\bR^2,\|\cdot\|_\P)$.
\end{proposition}

\begin{proof}
Suppose $(G,p)$ is well-positioned and infinitesimally rigid and let $G\to G'$ be a vertex-to-$K_4$ move on the vertex $v_0\in V(G)$ which introduces new vertices $v_1$, $v_2$ and $v_3$. 
Since $v_0$ has finite valence, there exists an open ball $B(p_{v_0},r)$ such that if $p_{v_0}$ is replaced with any point $x\in B(p_{v_0},r)$ then $(G,x)$ and $(G,p)$ induce the same framework colouring on $G$. 
Let $(K_4,\tilde{p})$ be the well-positioned and infinitesimally rigid placement of $K_4$ constructed in Example \ref{K4Rigid}. 
Define a well-positioned placement $p'$ of $G'$ by  setting $p'_{v}=p_{v}$ for all $v\in V(G)$ and  
\[p'_{v_1}=p_{v_0}+\epsilon\tilde{p}_{v_1}, \,\,\,\,\,\,\,\, p'_{v_2}=p_{v_0}+\epsilon\tilde{p}_{v_2}, \,\,\,\,\,\,\,\, p'_{v_3}=p_{v_0}+\epsilon\tilde{p}_{v_3}\]
where $\epsilon>0$ is chosen to be sufficiently small so that 
$p'_{v_1}$, $p'_{v_2}$ and $p'_{v_3}$ are each contained in $B(p_{v_0},r)$.
Suppose $u\in\F(G',p')$. By the argument in Example \ref{K4Rigid}, the restriction of $u$ to the vertices
$v_0,v_1,v_2,v_3$ is constant.
Thus if $v_0w$ is an edge of $G$ with framework colour $[F]$ which is replaced by $v_kw$ in $G'$ then applying Theorem \ref{flex1}, 
\[\varphi_{F}(u_{v_0}-u_w) = \varphi_F(u_{v_0}-u_{v_k})+\varphi_F(u_{v_k}-u_w)=0\]
and so the restriction of $u$ to $V(G)$ is an infinitesimal flex of $(G,p)$.
Since $(G,p)$ is infinitesimally rigid this restriction is constant and we conclude that $u$ is a trivial infinitesimal flex of $(G',p')$.
\end{proof}

We now show that the class of finite graphs which have minimally infinitesimally rigid well-positioned placements in $(\bR^2,\|\cdot\|_\P)$ is precisely the class of $(2,2)$-tight graphs. In particular, the existence of such a placement does not depend on the choice of polyhedral norm on $\bR^2$.

\begin{theorem}
\label{Laman}
Let $G$ be a finite simple graph and let $\|\cdot\|_\P$ be a polyhedral norm on $\bR^2$.
The following statements are equivalent.
\begin{enumerate}[(i)]
\item $G$ is minimally rigid in $(\bR^2,\|\cdot\|_\P)$.
\item $G$ is $(2,2)$-tight.
\end{enumerate}
\end{theorem}

\begin{proof}
$(i)\Rightarrow (ii)$. If $G$ is minimally rigid then there exists a placement $p$ such that $(G,p)$ is minimally infinitesimally rigid in $(\bR^2,\|\cdot\|_\P)$ and the result  follows from Proposition \ref{Maxwell}.

$(ii)\Rightarrow(i)$.
If  $G$ is $(2,2)$-tight then there exists a finite sequence of allowable graph moves,
\[K_1\overset{\mu_1}\longrightarrow G_2\overset{\mu_2}\longrightarrow G_3\overset{\mu_3}\longrightarrow\cdots \overset{\mu_{n-1}}\longrightarrow G\]
Every placement of $K_1$ is certainly infinitesimally rigid.
By Propositions \ref{Graph1}-\ref{Graph4}, for each graph in the sequence there exists a well-positioned and infinitesimally rigid placement  in $(\bR^2,\|\cdot\|_\P)$. 
In particular, $(G,p)$ is infinitesimally rigid for some well-positioned placement $p$. 
If a single edge is removed from $G$ then by Proposition \ref{Maxwell}, the resulting subframework is infinitesimally flexible. Hence $(G,p)$ is minimally infinitesimally rigid in $(\bR^2,\|\cdot\|_\P)$.  
\end{proof}

The collection of placements of a $(2,2)$-tight graph which are well-positioned and minimally infinitesimally rigid in $(\bR^2,\|\cdot\|_\P)$ varies with $\P$ and this is illustrated in the following two examples.  

\begin{example}
\label{NonRegEx}
Let $(G,p)$ be the well-positioned bar-joint framework in $(\bR^2,\|\cdot\|_\infty)$ illustrated in Figure \ref{NonRegFig}. The monochrome subgraph $G_{F_2}$ (indicated in gray) is not a spanning subgraph of $G$  and so, by Theorem \ref{RigThm}, $(G,p)$ is infinitesimally flexible. 
The graph $G$ is $(2,2)$-tight  and so, by Theorem \ref{Laman}, the regular placements of $G$ are infinitesimally rigid. We conclude that $(G,p)$ is not a regular bar-joint framework.
\end{example}

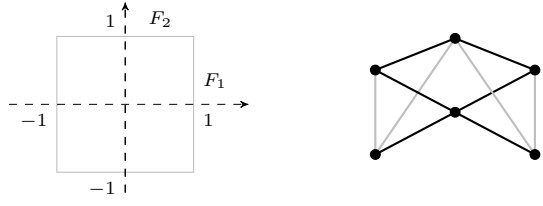
\begin{figure}[h]
\centering
  \begin{tabular}{  c   }
  
 \begin{minipage}{.45\textwidth}
\begin{tikzpicture}[scale=0.9, axis/.style={very thin,dashed, ->, >=stealth'},
    important line/.style={very thick},
    dashed line/.style={dashed,  thick},
    pile/.style={thick, ->, >=stealth', shorten <=2pt, shorten
    >=2pt},
    every node/.style={color=black}]
 
  \clip (-3.5,-1.5) rectangle (2cm,1.5cm); 
  
  \coordinate (A1) at (-1,1);
  \coordinate (A2) at (0,0);
  \coordinate (A3) at (1,-1);
  \coordinate (A4) at (1,1);
  \coordinate (A5) at (-1,-1); 

  \draw[lightgray] (A3) -- (A4) -- (A1) -- (A5) -- cycle;

  \draw[axis] (-1.7,0)  -- (1.8,0) node(xline)[right]
        {$ $};
  \draw[axis] (0,-1.3) -- (0,1.5) node(yline)[above] {$ $};

  \draw[dashed] (0,-1.3) -- (0,1.5);

  \node[below left] at (-1,0)  {\tiny $-1$};
  \node[below right] at (1,0)  {\tiny $1$};	

  \node[above left] at (0,1)  {\tiny $1$};	
  \node[below left] at (0,-1)  {\tiny $-1$};	

  \node[above right] at (0.2,1)  {\tiny $F_2$};	
  \node[above right] at (1,0.1)  {\tiny $F_1$};	
\end{tikzpicture}
\end{minipage}

    \begin{minipage}{.45\textwidth}
 \begin{tikzpicture}[scale=0.7]
 
    \clip (-1.4,-0.5) rectangle (3cm,2.4cm); 
  
  \coordinate (Aone) at (-0.5,0);
  \coordinate (Atwo) at (1,0.8);
  \coordinate (Athree) at (2.5,0);
 
  \coordinate (Afive) at (2.5,1.6);
  \coordinate (Asix) at (1,2.2);
  \coordinate (Aseven) at (-0.5,1.6);

 \draw[thick] (Aone) -- (Atwo);
 \draw[thick] (Atwo) -- (Athree);
 \draw[thick] (Afive) -- (Asix);
 \draw[thick]  (Asix) -- (Aseven);

 \draw[thick,lightgray] (Aone) -- (Asix); 
\draw[thick,lightgray] (Aone) -- (Aseven); 

 \draw[thick] (Atwo) -- (Afive); 
 \draw[thick] (Atwo) -- (Aseven); 

 \draw[thick,lightgray] (Athree) -- (Afive); 
 \draw[thick,lightgray] (Athree) -- (Asix); 


      \node[draw,circle,inner sep=1.3pt,fill] at (Aone) {};
  \node[draw,circle,inner sep=1.3pt,fill] at (Atwo) {};
  \node[draw,circle,inner sep=1.3pt,fill] at (Athree) {};

  \node[draw,circle,inner sep=1.3pt,fill] at (Afive) {};
  \node[draw,circle,inner sep=1.3pt,fill] at (Asix) {};
  \node[draw,circle,inner sep=1.3pt,fill] at (Aseven) {};

  \end{tikzpicture}

\end{minipage}

\end{tabular}
 \caption{A non-regular, infinitesimally flexible placement of a $(2,2)$-tight graph in $(\bR^2,\|\cdot\|_\infty)$}
\label{NonRegFig}
\end{figure}

In the following example we consider the same bar-joint framework as in Example \ref{NonRegEx} 
but with a different polyhedral norm. In this case the placement is infinitesimally rigid.

\begin{example}
A large class of polyhedral norms can be derived from submodular functions. 
Let $S=\{1,2,\ldots,d\}$ and let $f:2^S\to\bR$ be a function on the power set of $S$.
For each $x=(x_1,x_2,\ldots,x_d)\in \bR_+^d$ define 
$A_k(x)=\{j_1,\ldots,j_k\}\subseteq S$
where $x_{j_1}\geq x_{j_2}\geq\cdots\geq x_{j_d}\geq 0$.
The Lov\'{a}sz extension of $f$ (\cite{Lovasz}) is defined for $x\in \bR^d_+$ by,
\[\hat{f}(x) = \sum_{k=1}^d x_{j_k}\triangle_kf(x)\]
where $\triangle_kf(x) = f(A_k(x))-f(A_{k-1}(x))$.
If $f$ is submodular, monotone and satisfies $f(\emptyset)=0$ and $f(\{j\})>0$ for each $j$ then the function
\[\|x\|_\P = \hat{f}(|x_1|,\ldots,|x_d|)\]
is a polyhedral norm on $\bR^d$.
If $F$ is a facet of $\P$ and $x\in \cone(F)^\circ$ then
\[\hat{F} =  
(\sgn(x_1)\triangle_{\sigma(1)}f(|x|),\ldots,\sgn(x_d)\triangle_{\sigma(d)}f(|x|))
\]
where $\sigma:\{1,\ldots,d\}\to\{1,\ldots,d\}$ is the inverse of the permutation 
$k\mapsto j_k$ determined by the coordinates of $x$.
Consider, for example,  the submodular function $f:2^S\to \bR$ where  $S=\{1,2\}$ and,
\[f(A) =\left\{\begin{array}{ll}
0 & \mbox{ if } A=\emptyset \\
1 & \mbox{ if } A=\{2\} \\
2 & \mbox{ otherwise }
\end{array}\right.
 \]
The associated polyhedral norm is defined for $x=(x_1,x_2)\in\bR^2$ by,
\[\|x\|_\P = \left\{\begin{array}{ll}
2|x_1| & \mbox{ if } |x_1|\leq |x_2| \\
|x_1|+|x_2| & \mbox{ if } |x_1|\geq |x_2|
\end{array}\right.\]
Let $(G,p)$ be the bar-joint framework in $(\bR^2,\|\cdot\|_\P)$ illustrated in Figure \ref{SubmodularFig}. The monochrome subgraphs  induced by the facets $F_1$, $F_2$ and $F_3$ are indicated in black, gray and dashed lines respectively and the corresponding extreme points of the polar set $\P^\triangle$ are,
\[\hat{F}_1=(1,1), \,\,\,\,\,\,\, \hat{F}_2 = (1,0), \,\,\,\,\,\,\,\, \hat{F}_3=(-1,1)\] 
The rigidity matrix $R_\P(G,p)$ is, 
\[ \kbordermatrix{
& a_x & a_y && b_x & b_y &   & c_x & c_y  && d_x & d_y& &  e_x & e_y& &  f_x & f_y \\
(ab,F_1)& 1  & 1 & \vrule & -1 & -1 & \vrule & 0 & 0 &\vrule&  0 & 0   &\vrule&  0 & 0 & \vrule & 0 & 0  \\
(ae,F_2)& 1 & 0 & \vrule & -1 & 0 &  \vrule &  0 & 0  &\vrule &0 & 0  &\vrule&  0 & 0 & \vrule & 0 & 0  \\
(af,F_2)& 1  & 0 & \vrule & 0 & 0 &  \vrule &   -1 & 0  &\vrule & 0 & 0  & \vrule & 0 & 0  &\vrule & 0 & 0 \\
(bc,F_3)&0 & 0 &\vrule & -1 & 1 &  \vrule &  1& -1 &\vrule & 0 & 0  &\vrule&  0 & 0  &\vrule & 0 & 0  \\
(bd,F_1)&0 & 0  &\vrule  &1  & 1 &  \vrule & 0 & 0  &\vrule &-1  & -1  &\vrule & 0 & 0  &\vrule & 0 & 0 \\
(bf,F_3)& 0 & 0  &\vrule & -1 & 1 &  \vrule &  0 & 0  &\vrule &0 &0&  \vrule&  0 & 0  &\vrule &1 & -1 \\
(cd,F_2)& 0 & 0 & \vrule &0  & 0 & \vrule & 1 & 0   &\vrule &  -1 & 0  &\vrule&  0 & 0  &\vrule & 0 & 0 \\
(ce,F_2) &0 & 0  &\vrule & 0 & 0 &  \vrule & 1 & 0  &\vrule & 0 &0& \vrule&  -1 & 0   &\vrule & 0 & 0  \\
(de,F_3)& 0 & 0  & \vrule& 0 & 0 &  \vrule & 0  & 0 &\vrule &  -1 & 1  &\vrule &1 & -1  &\vrule & 0 & 0  \\
(ef,F_1)& 0 & 0  & \vrule & 0 & 0 & \vrule &0 & 0 &\vrule & 0 &0& \vrule&  1  & 1 &\vrule & -1  & -1 
}\]
By computing the rank of the rigidity matrix or alternatively by applying the edge-labelled path argument of Proposition \ref{RelRigid} we see that $(G,p)$ is minimally infinitesimally rigid.

\end{example}

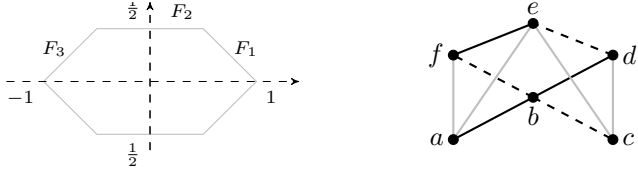
\begin{figure}[h]
\centering
  \begin{tabular}{  c   }
  
 \begin{minipage}{.45\textwidth}
\begin{tikzpicture}[scale=0.7, axis/.style={very thin,dashed, ->, >=stealth'},
    important line/.style={very thick},
    dashed line/.style={dashed,  thick},
    pile/.style={thick, ->, >=stealth', shorten <=2pt, shorten
    >=2pt},
    every node/.style={color=black}]
 
  \clip (-3.5,-1.7) rectangle (3cm,1.5cm); 
  
  \coordinate (A1) at (-1,-1);
  \coordinate (A2) at (1,-1);
  \coordinate (A3) at (2,0);
  \coordinate (A4) at (1,1);
  \coordinate (A5) at (-1,1); 
 \coordinate (A6) at (-2,0); 

  \draw[lightgray] (A1) -- (A2) -- (A3) -- (A4) -- (A5) -- (A6) -- cycle;

  \draw[axis] (-2.7,0)  -- (2.8,0) node(xline)[right]
        {$ $};
  \draw[axis] (0,-1.3) -- (0,1.5) node(yline)[above] {$ $};

  \draw[dashed] (0,-1.3) -- (0,1.5);

  \node[below left] at (-2,0)  {\tiny $-1$};
  \node[below right] at (2,0)  {\tiny $1$};	

  \node[above left] at (0,1)  {\tiny $\frac{1}{2}$};	
  \node[below left] at (0,-1)  {\tiny $\frac{1}{2}$};	

  \node[above right] at (0.2,1)  {\tiny $F_2$};	
  \node[above right] at (1.4,0.3)  {\tiny $F_1$};	
 \node[above left] at (-1.4,0.3)  {\tiny $F_3$};	
\end{tikzpicture}
\end{minipage}

    \begin{minipage}{.45\textwidth}
 \begin{tikzpicture}[scale=0.7]
 
    \clip (-1.4,-0.8) rectangle (3cm,2.8cm); 
  
  \coordinate (Aone) at (-0.5,0);
  \coordinate (Atwo) at (1,0.8);
  \coordinate (Athree) at (2.5,0);
 
  \coordinate (Afive) at (2.5,1.6);
  \coordinate (Asix) at (1,2.2);
  \coordinate (Aseven) at (-0.5,1.6);

 \draw[thick] (Aone) -- (Atwo);
 \draw[thick,dashed] (Atwo) -- (Athree);
 \draw[thick,dashed] (Afive) -- (Asix);
 \draw[thick]  (Asix) -- (Aseven);

 \draw[thick,lightgray] (Aone) -- (Asix); 
\draw[thick,lightgray] (Aone) -- (Aseven); 

 \draw[thick] (Atwo) -- (Afive); 
 \draw[thick,dashed] (Atwo) -- (Aseven); 

 \draw[thick,lightgray] (Athree) -- (Afive); 
 \draw[thick,lightgray] (Athree) -- (Asix); 


      \node[draw,circle,inner sep=1.3pt,fill] at (Aone) {};
  \node[draw,circle,inner sep=1.3pt,fill] at (Atwo) {};
  \node[draw,circle,inner sep=1.3pt,fill] at (Athree) {};

  \node[draw,circle,inner sep=1.3pt,fill] at (Afive) {};
  \node[draw,circle,inner sep=1.3pt,fill] at (Asix) {};
  \node[draw,circle,inner sep=1.3pt,fill] at (Aseven) {};

 \node[left] at (Aone)  { $a$};	
  \node[below] at (Atwo)  { $b$};	
 \node[right] at (Athree)  {$c$};	
\node[left] at (Aseven)  { $f$};	
  \node[above] at (Asix)  { $e$};	
 \node[right] at (Afive)  {$d$};	
  \end{tikzpicture}

\end{minipage}
\end{tabular}
 \caption{A minimally infinitesimally rigid bar-joint framework in $(\bR^2,\|\cdot\|_\P)$}
\label{SubmodularFig}
\end{figure}

\section{Infinite frameworks}
\label{Infinite}
In this section we consider some aspects of rigidity which are unique to infinite bar-joint frameworks. 
Let 
$B(p_v,r)$ be the open ball in $\bR^d$ with centre $p_v$ and radius $r$.
An infinite  bar-joint framework $(G,p)$ is {\em uniformly well-positioned} in $(\bR^d,\|\cdot\|_\P)$ if there exists $r>0$ such that 
$(G,x)$ is well-positioned for all $x\in \prod_{v\in V(G)} B(p_v,r)$.
An {\em equicontinuous flex} of $(G,p)$ is a finite flex $\{\alpha_v:v\in V(G)\}$ which is also equicontinuous as a collection of functions from an interval $(-\delta,\delta)$ into $\bR^d$.
An infinite graph is locally finite if every vertex has finite valence.



\begin{proposition}\label{Equicontinuous}
Let $(G,p)$ be a uniformly well-positioned bar-joint framework in $(\bR^d,\|\cdot\|_\P)$ and suppose that $G$ is locally finite.
If $(G,p)$ is infinitesimally rigid then every equicontinuous flex of $(G,p)$ is trivial.
\end{proposition}

\begin{proof}
Since $(G,p)$ is uniformly well-positioned there exists $r>0$ such that
$(G,x)$ is well-positioned for all $x\in U:= \prod_{v\in V(G)} B(p_v,r)$.
In particular, $x_v-x_w$ is contained in the conical hull of the same unique facet $F_{vw}$ as $p_v-p_w$ 
for each edge $vw\in E(G)$.
Hence $\|x_v-x_w\|_\P = (x_v-x_w)\cdot\hat{F}_{vw}$ for all $x\in U$ and all $vw\in E(G)$.
As in Proposition  \ref{Configuration} we have 
\[V(G,p)\cap U = (p+\F(G,p))\cap U\]
where we now regard $U$ as an open neighbourhood of $p$ with respect to the box topology on $\prod_{v\in V(G)} \bR^d$.
If $\{\alpha_v:v\in V(G)\}$ is an equicontinuous flex of $(G,p)$ then there exists $\delta>0$ such that
$\alpha_v(t)\in B(p_v,r)$ for all $|t|<\delta$.
Thus $(\alpha_v(t))_{v\in V(G)}\in V(G,p)\cap U$ for all $|t|<\delta$.
Now $(\alpha_v(t)-p_v)_{v\in V(G)}\in \F(G,p)$ is an infinitesimal flex of $(G,p)$ for each $|t|<\delta$ and so must be trivial.
Hence there exists $c(t)\in \bR^d$ such that $\alpha_v(t)-p_v=c(t)$ for all $v\in V(G)$ and all $|t|<\delta$. The function $c:(-\delta,\delta)\to\bR^d$ is continuous and so
$\{\alpha_v:v\in V(G)\}$ is a trivial finite flex of $(G,p)$.
\end{proof}



A {\em vertex-complete tower of bar-joint frameworks} in $(G,p)$ is a sequence of finite subframeworks $\{(G_k,p):k\in\bN\}$ such that $G_k$ is a subgraph of $G_{k+1}$ for each $k\in\bN$ and $V(G)=\cup_{k\in\bN} V(G_k)$.
In \cite[Theorem 3.10]{kit-pow-2} it is shown that given any norm on $\bR^d$, a countable bar-joint framework is infinitesimally rigid if and only if it contains a vertex-complete tower such that $(G_k,p)$ is relatively infinitesimally rigid in $(G_{k+1},p)$ for each $k\in\bN$.
The following is a direct proof of this fact for polyhedral norms which exploits the edge-labelling methods of Section \ref{Edge}.

\begin{proposition}
\label{Towers}
 Let $(G,p)$ be a countable bar-joint framework in $(\mathbb{R}^d,\|\cdot\|_\P)$
and suppose that $|\Phi(G,p)|=d$.
Then the following statements are equivalent.
\begin{enumerate}[(i)]
\item $(G,p)$ is infinitesimally rigid.
\item $(G,p)$ has a vertex-complete tower of relatively infinitesimally rigid subframeworks.
\end{enumerate}
\end{proposition}

\begin{proof}
$(i)\Rightarrow (ii)$
Choose a vertex-complete tower of bar-joint frameworks $\{(G_k,p):k\in\bN\}$ in $(G,p)$
and let $H_1=G_1$.
By Theorem \ref{RigThm}, $G_F$ is a spanning subgraph of $G$ for each framework colour $[F]\in\Phi(G,p)$.
It follows that if $v,w\in V(H_1)$ then there exists a path in $G_F$ from $v$ to $w$ for each  $[F]\in \Phi(G,p)$.
Let $H_2$ be the subgraph of $G$ spanned by the union of $G_2$ with these finitely many paths. 
By Proposition \ref{RelRigid}, $(H_1,p)$ is relatively infinitesimally rigid in $(H_2,p)$.
If $v,w\in V(H_2)$ then there exists a path in $G_F$ from $v$ to $w$ for each $[F]\in \Phi(G,p)$.
Let $H_3$ be the subgraph of $G$ spanned by the union of $G_3$ with these finitely many paths. 
Continuing this process we construct the desired framework tower $\{(H_k,p):k\in\bN\}$.

$(ii)\Rightarrow (i)$
Suppose there exists a vertex-complete tower $\{(G_k,p):k\in\bN\}$ of relatively infinitesimally rigid subframeworks in $(G,p)$.
If $(G,p)$ is not infinitesimally rigid then by Theorem \ref{RigThm} 
there exists a vertex $v_0\in V(G)$ and a facet $F$ of $\P$ such that $v_0\notin V(G_F)$.
Now $v_0\in V(G_k)$ for some $k\in\bN$ where $|V(G_k)|\geq2$.
Since $|\Phi(v_0)|<d$ 
we may define an infinitesimal flex $u$ of $(G_{k+1},p)$ as in the proof of Proposition \ref{Vertex}.
Thus we set $u_{v_0}=x\not=0$ and $u_v=0$ for all $v\in V(G_{k+1})\backslash\{v_0\}$.
The restriction of $u$ to $(G_{k},p)$ is  non-trivial which is a contradiction.
We conclude that $(G,p)$ is infinitesimally rigid.
\end{proof}

A bar-joint framework is {\em sequentially infinitesimally rigid} if it contains a 
vertex-complete tower of bar-joint frameworks $\{(G_k,p):k\in\bN\}$
such that $(G_k,p)$ is infinitesimally rigid for each $k\in\bN$.
It is shown in \cite{kit-pow-2} that infinitesimal rigidity and sequential infinitesimal rigidity are equivalent properties for all generic bar-joint frameworks in Euclidean space $\bR^2$, and more generally in 
$(\bR^2,\|\cdot\|_p)$ for all $\ell^p$ norms with $p\in(1,\infty)$.
The following example shows that sequential infinitesimal rigidity is in general not equivalent to infinitesimal rigidity for countable bar-joint frameworks  in $(\mathbb{R}^2,\|\cdot\|_\P)$.

\begin{example}
Let $(G,p)$ be the well-positioned countable bar-joint framework in $(\mathbb{R}^2,\|\cdot\|_\infty)$ which is illustrated in Figure \ref{Framework1}.
The monochrome subgraphs $G_{F_1}$ and $G_{F_2}$ indicated in black and gray respectively
are both spanning trees in $G$ and so $(G,p)$ is minimally infinitesimally rigid by Corollary \ref{MinRigCor}. A vertex-complete tower of relatively infinitesimally rigid subframeworks in $(G,p)$ is evident by letting $G_k$ be  the subgraph of $G$ induced by the vertices $v_0,v_1,\ldots,v_{k+1}$ and applying Proposition \ref{RelRigid} (cf. Proposition \ref{Towers}). If $(H,p)$ is a finite subframework of $(G,p)$ with $V(H)=\{v_{j_1},v_{j_2},\ldots,v_{j_k}:j_1<j_2<\cdots<j_k\}$ then $|\Phi(v_{j_k})|<2$ and so $(H,p)$ is infinitesimally flexible by Proposition \ref{Vertex}. 
In particular, there does not exist a vertex-complete tower of infinitesimally rigid finite subframeworks and so $(G,p)$ is not sequentially infinitesimally rigid. Note that $(G,p)$ is uniformly well-positioned and the graph is locally finite. Hence by Proposition \ref{Equicontinuous}, the  equicontinuous finite flexes of $(G,p)$ are necessarily trivial.
\end{example}

 \begin{figure}[ht]
\centering
  \begin{tabular}{  c   }
  
 \begin{minipage}{.45\textwidth}
\begin{tikzpicture}[scale=0.9, axis/.style={very thin,dashed, ->, >=stealth'},
    important line/.style={very thick},
    dashed line/.style={dashed,  thick},
    pile/.style={thick, ->, >=stealth', shorten <=2pt, shorten
    >=2pt},
    every node/.style={color=black}]
 
  \clip (-3.5,-1.5) rectangle (2cm,1.5cm); 
  
  \coordinate (A1) at (-1,1);
  \coordinate (A2) at (0,0);
  \coordinate (A3) at (1,-1);
  \coordinate (A4) at (1,1);
  \coordinate (A5) at (-1,-1); 

  \draw[lightgray] (A3) -- (A4) -- (A1) -- (A5) -- cycle;
    
  \draw[axis] (-1.7,0)  -- (1.8,0) node(xline)[right]
        {$ $};
  \draw[axis] (0,-1.3) -- (0,1.5) node(yline)[above] {$ $};

  \draw[dashed] (0,-1.3) -- (0,1.5);

  \node[below left] at (-1,0)  {\tiny $-1$};
  \node[below right] at (1,0)  {\tiny $1$};	

  \node[above left] at (0,1)  {\tiny $1$};	
  \node[below left] at (0,-1)  {\tiny $-1$};	

  \node[above right] at (0.2,1)  {\tiny $F_2$};	
  \node[above right] at (1,0.1)  {\tiny $F_1$};	
\end{tikzpicture}
\end{minipage}

    \begin{minipage}{.45\textwidth}
  \begin{tikzpicture}[scale=0.3]
 
    \clip (-4.1,-5) rectangle (8.6cm,7.2cm); 
  
  \coordinate (A1) at (7,6);
  \coordinate (A2) at (7,5);
  \coordinate (A3) at (7,3);
  \coordinate (A4) at (7,-3);
  \coordinate (A5) at (6,5);
  \coordinate (A6) at (4,3);
  \coordinate (A7) at (-2,-3);

 \draw[thick, lightgray] (A1) -- (A2) -- (A3) -- (A4);	
 \draw[thick]  (A1) -- (A5) -- (A6) -- (A7);	
 \draw[thick]  (A2) -- (A5);	
 \draw[thick]  (A3) -- (A6);	
 \draw[thick]  (A4) -- (A7);	
 \draw[thick, lightgray] (A3) -- (A5);	
 \draw[thick, lightgray] (A6) -- (A4);

  \node[draw,circle,inner sep=1.1pt,fill] at (A1) {};
  \node[draw,circle,inner sep=1.1pt,fill] at (A2) {};
  \node[draw,circle,inner sep=1.1pt,fill] at (A3) {};
 \node[draw,circle,inner sep=1.1pt,fill] at (A4) {};
  \node[draw,circle,inner sep=1.1pt,fill] at (A5) {};
  \node[draw,circle,inner sep=1.1pt,fill] at (A6) {};
  \node[draw,circle,inner sep=1.1pt,fill] at (A7) {};
   \node[below] at (7,-2.4) {$\vdots$};
  \node[below left] at (-2.3,-2.4) {$\iddots$};
 
  \node[right] at (7,6.3)  {\small $v_0$};
  \node[right] at (A2)  {\small $v_1$};	
  \node[right] at (A3)  {\small $v_3$};
  \node[right] at (A4)  {\small $v_5$};	
   \node[left] at (A5)  {\small $v_2$};	
  \node[left] at (A6)  {\small $v_4$};
  \node[left] at (A7)  {\small $v_6$};	
  
  \end{tikzpicture}
\end{minipage}
\end{tabular}

\caption{A countable bar-joint framework which is infinitesimally rigid in $(\mathbb{R}^2,\|\cdot\|_\infty)$ but not sequentially infinitesimally rigid.}
\label{Framework1}
 \end{figure}
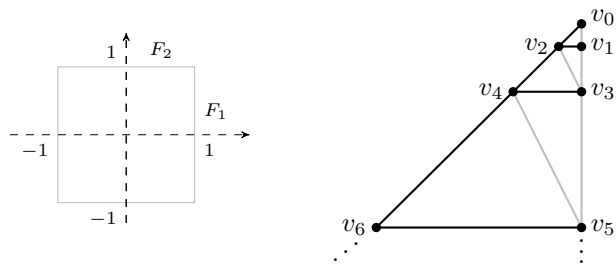










\bibliographystyle{abbrv}

\begin{thebibliography}{30}


\bibitem{asi-rot} Asimow, L., Roth, B.:
{The rigidity of
graphs}. Trans. Amer. Math. Soc., {\bf 245},  279-289 (1978)



\bibitem{asi-rot-2}
 Asimow, L., Roth, B.: The rigidity of graphs II. J. Math. Anal. Appl.
{\bf 68}, 171-190 (1979)


\bibitem{alex}
Alexandrov, A. D.:  Konvexe Polyeder. Akademie-Verlag, Berlin (1958)

\bibitem{bor-str} Borcea, C.S., Streinu, I.: Periodic frameworks and flexibility.
     Proc. R. Soc. A {\bf 466}, 2633-2649  (2010)




\bibitem{con-1} Connelly, R.: Generic global rigidity. Discrete Comput. Geom. {\bf 33}(4), 549-563  (2005)

\bibitem{cau} Cauchy, A.: Sur les polygones et poly\`{e}dres. Second M\'{e}moir. J \'{E}cole Polytechn. 9 (1813) 87-99; Oeuvres. T. 1. Paris 1905, pp. 26-38.








\bibitem{gor-hea-thu} Gortler, S., Healy, A., Thurston, D.: Characterizing generic global rigidity. American Journal of Mathematics {\bf 132}(4),  897-939 (2010)

\bibitem{gra-ser-ser} Graver,  J., Servatius, B., Servatius, H.:
{Combinatorial rigidity}. Graduate Texts in Mathematics,
vol 2, Amer. Math. Soc. (1993)

\bibitem{grun}
Gr\"{u}nbaum, B.: Convex polytopes. Pure and Applied Mathematics, Vol. 16 Interscience Publishers John Wiley \& Sons, Inc., New York (1967)


\bibitem{jac-jor} Jackson, B., Jordan, T.: Connected rigidity matroids and unique realisations of graphs. J. Combinatorial Theory B, {\bf 94},  1-29 (2005)













\bibitem{kit-pow} Kitson, D.,  Power, S.C.: Infinitesimal rigidity for non-Euclidean bar-joint frameworks. Bull. London Math. Soc., (to appear). http://arxiv.org/abs/1304.3385

\bibitem{kit-pow-2} Kitson, D., Power, S.C.: The rigidity of infinite graphs. Preprint 2013.
http://arxiv.org/abs/1310.1860

\bibitem{Lam} Laman, G.:  On graphs and the rigidity of plane skeletal structures. J. Engineering Mathematics, {\bf 4},  331-340 (1970)

\bibitem{Lovasz} Lov\'{a}sz, L.:
Submodular functions and convexity. In: Mathematical programming: the state of the art (Bonn, 1982), 235-257, Springer, Berlin, (1983) 

\bibitem{mal-the} Malestein,  J.,  Theran, L.: Generic combinatorial rigidity of periodic frameworks. Adv. Math. {\bf 233},  291-331 (2013)

\bibitem{max} Maxwell, J.C.:
On the calculation of the equilibrium and stiffness of frames.
Philosophical Magazine {\bf 27},  294-299 (1864)




\bibitem{NOP} Nixon, A.,  Owen, J.C., Power, S.C.: Rigidity of frameworks supported on surfaces. SIAM J. Discrete Math. {\bf 26}  1733-1757 (2012)

\bibitem{NOP2} Nixon, A.,  Owen, J.C., Power, S.C.:  A Laman theorem for frameworks on surfaces of revolution. preprint 2012. http://arxiv.org/abs/1210.7073









\bibitem{pow-poly}  Power, S.C.: Polynomials for crystal frameworks and the rigid unit mode spectrum. Royal Society Philosophical Transactions A, {\bf  372} (2014). http://arxiv.org/abs/1102.2744

\bibitem{ros-kavli}  Ross, E.: The rigidity of periodic body-bar frameworks on a fixed torus. Royal Society Philosophical Transactions A, {\bf  372} (2014). http://arxiv.org/abs/1202.6652



\bibitem{schulze} Schulze, B.: Symmetric versions of Laman's theorem. Discrete Comput. Geom. {\bf 44}(4), 946-972 (2010)








\bibitem{whi-union}  Whiteley, W.: The union of matroids and the rigidity of frameworks. SIAM J. Disc. Math., {\bf 1}(2),  237-255 (1988)







\bibitem{whi-1} Whiteley, W.: Infinitesimally rigid polyhedra. I. Statics of frameworks. Trans. Amer. Math. Soc. {\bf 285}(2), 431-465 (1984)

\bibitem{whi-2}  Whiteley, W.: Matroids and rigid structures. Matroid Applications, 1-53, 
Encyclopedia Math. Appl., 40, Cambridge Univ. Press, Cambridge, (1992)



\end{thebibliography}
\def\lfhook#1{\setbox0=\hbox{#1}{\ooalign{\hidewidth
  \lower1.5ex\hbox{'}\hidewidth\crcr\unhbox0}}}

\end{document}